\documentclass[11pt,a4paper]{amsart}[2000]
\title{Rokhlin dimension: obstructions and permanence properties}
\author{Ilan Hirshberg}
\address{Department of Mathematics, Ben Gurion University of the Negev, P.O.B.
653, Be'er Sheva 84105, Israel}

\author{N. Christopher Phillips}
\address{Department of Mathematics,
University of Oregon,
Eugene OR 97403-1222,
U.S.A.}
\thanks{This research was supported in part by the US-Israel Binational Science
Foundation. This material is partially based on work
  of the second author supported by the
  US National Science Foundation under Grant DMS-1101742.}

\usepackage{amssymb}
\usepackage{amsthm}
\usepackage[all]{xypic}
\usepackage{enumitem}

\theoremstyle{plain}

\newtheorem{Thm}{Theorem}[section]
\newtheorem{Cor}[Thm]{Corollary}
\newtheorem{Lemma}[Thm]{Lemma}

\newtheorem{Prop}[Thm]{Proposition}

\theoremstyle{definition}
\newtheorem{Def}[Thm]{Definition}
\newtheorem{Notation}[Thm]{Notation}

\newtheorem{Exl}[Thm]{Example}
\newtheorem{Rmk}[Thm]{Remark}

\newcommand{\B}{B}
\newcommand{\A}{A}
\newcommand{\J}{J}
\newcommand{\K}{\mathcal{K}}
\newcommand{\D}{D}
\newcommand{\Ch}{D}
\newcommand{\Zh}{\mathcal{Z}}
\newcommand{\E}{E}
\newcommand{\Oh}{\mathcal{O}}

\newcommand{\T}{{\mathbb T}}
\newcommand{\R}{{\mathbb R}}
\newcommand{\N}{{\mathbb N}}
\newcommand{\Z}{{\mathbb Z}}
\newcommand{\C}{{\mathbb C}}
\newcommand{\Q}{{\mathbb Q}}

\newcommand{\aut}{\mathrm{Aut}}

\newcommand{\eps}{\varepsilon}
\numberwithin{equation}{section}

\newcommand{\dimrokct}{\dim^{\mathrm{c}}_{\mathrm{Rok}}}
\newcommand{\dr}{\mathrm{dr}}
\newcommand{\id}{\mathrm{id}}
\newcommand{\dimnuc}{\mathrm{dim}_{\mathrm{nuc}}}
\newcommand{\supp}{\mathrm{supp}}

\begin{document}
\begin{abstract}
This paper is a further study of finite Rokhlin dimension for actions of finite
groups and the integers on $C^*$-algebras, introduced by the first author,
Winter, and Zacharias. We extend the definition of finite Rokhlin dimension to
the nonunital case. This definition behaves well with respect to extensions,
and is sufficient to establish permanence of finite nuclear dimension and
$\Zh$-absorption. We establish $K$-theoretic obstructions to the existence of
actions of finite groups with finite Rokhlin dimension (in the commuting tower
version). In particular, we show that there are no actions of any nontrivial
finite group on the Jiang-Su algebra or on the Cuntz algebra $\Oh_{\infty}$ with
finite Rokhlin dimension in this sense. 
\end{abstract}
\maketitle

The study of group actions on $C^*$-algebras, and their associated
crossed products, has always been a central research theme in
operator algebras.
One would like to identify properties of
group actions which on the one hand occur commonly and naturally
enough to be of interest, and on the other hand are strong enough
to be used to derive interesting properties of the action or of
the crossed product. Examples of important properties for a group action meeting
these criteria are the various forms of the Rokhlin property, which arose early
on in the theory. See, for instance, \cite{Izumi-survey} and references therein
for actions of $\Z$ and \cite{Izumi-I,Izumi-II,Phillips-survey,OP} for the finite group case.
The Rokhlin property for the single automorphism case is quite prevalent, and generic in some cases, forming a dense $G_{\delta}$ set in the
automorphism group (see \cite{HWZ}). 
However, it requires the existence of projections, and thus will not occur in
cases of interest which have few projections, such as automorphisms of the
Jiang-Su algebra $\Zh$ or automorphisms arising from topological dynamical
systems on connected spaces. 

Rokhlin dimension was introduced in \cite{HWZ} as a generalization of the
Rokhlin property, motivated by
the definition of covering dimension for topological spaces. In this formulation, the Rokhlin property becomes Rokhlin
dimension 0. In the definition of higher Rokhlin dimension, the projections from
the Rokhlin property are replaced by positive elements with controlled overlaps. 
This generalizations covers many more cases. It is shown in \cite{HWZ} that for
separable unital $\Zh$-absorbing $C^*$-algebras, the property of having Rokhlin
dimension at most 1 is generic. In the commutative setting, it was shown that if $X$ is a
compact metrizable space of finite covering dimension, $h \colon X \to X$ is a minimal
homeomorphism, and $\alpha \in \aut(C(X))$ is given by $\alpha(f) = f \circ h^{-1}$,
then $\alpha$ has finite Rokhlin dimension. The result concerning
homeomorphisms was generalized recently in \cite{szabo} to the case of
free actions of $\Z^m$ on finite dimensional spaces.

Actions of finite groups with the Rokhlin property are much less common. As in
the case of a single automorphism, it requires projections, thereby ruling out
 actions on $\Zh$ with the Rokhlin property. Even when
there are many projections, there are simple $K$-theoretic
obstructions to the existence of actions with the Rokhlin property. For instance, since any automorphism of $\Oh_{\infty}$ acts trivially on $K_0$, if $\alpha$ is an action of a finite group $G$ on
$\Oh_{\infty}$ with the Rokhlin property, and $(e_g)_{g \in G}$ is a family of Rokhlin
projections, then $[e_g] = [e_h]$ in $K_0(\Oh_{\infty})$ for all $g,h \in G$. Thus,
$[1]$ is divisible by the order of the group, $\# G$. This cannot happen if $G$ has more than one
element. Likewise, one can see that there are no actions of $\Z_p = \Z/p\Z$ on the UHF algebra 
$M_{q^{\infty}}$ with the Rokhlin property if $p$ does not divide some power of $q$. For more, we refer the reader to \cite[Example 3.12]{Phillips-survey} and the discussion after it.

This paper is devoted to a further study of Rokhlin dimension, mainly for the
finite group case.
In Sections
\ref{sec:finite-group-permanence-properties} and
\ref{sec:single-auto-permanence-properties} we generalize Rokhlin dimension to the
nonunital case. Our definition is sufficient for generalizing the results concerning
permanence of finite nuclear dimension and decomposition rank
(\cite{kirchberg-winter,winter-zacharias}) and $\Zh$-absorption, and behaves
well with respect to extensions. In Section \ref{sec:obstructions} we study $K$-theoretic
obstructions to finite Rokhlin dimension.
The $K$-theoretic obstructions here are more subtle than the ones described
above, and involve the structure of equivariant $K$-theory viewed as a
module over the representation ring. As a consequence we show, for instance,
that there are no actions of (nontrivial) finite groups on $\Zh$ or
$\Oh_{\infty}$ with the commuting tower version of finite Rokhlin dimension.
There are, however, natural examples of actions of finite group
actions on $C^*$-algebras which do not have the Rokhlin property but do have
finite Rokhlin dimension; see for instance Example
\ref{Exl:irrational-rotation-algebra}. The distinction between the commuting tower and noncommuting tower
versions of Rokhlin dimension initially appeared to be a minor technicality. However,
it was recently shown in \cite[Theorem 2.3]{BEMSW} that any outer action of $\Z_2$ on $\Oh_{\infty}$ has
Rokhlin dimension 1 in the noncommuting tower sense. This sharply contrasts with the results we present in Section \ref{sec:obstructions}. Likewise, the action of $S_n$ by permutation on the tensor factors of $\Zh \cong \Zh^{\otimes n}$ (see \cite{HW-permutations}) does not have finite Rokhlin dimension with commuting towers, but recently it has been shown (\cite[Proposition 6.10]{SWZ}) that this action has Rokhlin dimension 1 without commuting towers. Our results furthermore show that for
finite group actions, there is indeed a genuine difference between the commuting tower version of finite
Rokhlin dimension and the various projectionless versions of the tracial Rokhlin
property (\cite{Sato,hirshberg-orovitz}).

We use the following notational conventions throughout. We write $\# G$ for the number of elements in a group $G$. We write $\Z_p = \Z/p\Z$, since the $p$-adic numbers make no appearance in the paper. Order zero maps are always assumed to be completely positive (although not necessarily contractive). 

\section{Preliminaries}
\label{sec:preliminaries}

 We recall the definition of Rokhlin dimension from
\cite{HWZ}.
\begin{Def}
\label{def: positive Rokhlin finite groups}
Let $G$ be a finite group, let $\A$ be a unital $C^*$-algebra and let $\alpha \colon G \to
\textup{Aut} (\A)$ be an action of $G$ on $\A$.
We say that $\alpha$ has Rokhlin dimension $d$ with commuting towers, and write
$\dimrokct(\alpha)=d$, if $d$ is the least integer such that the following
holds. For any 
$\eps >0$ and every finite subset $F \subset \A$ there is a family 
$\left( f^{(l)}_g \right)_{l=0,1,\ldots,d\, ; g \in G}$ of positive contractions in $\A$
such that:
\begin{enumerate}
\item \label{def-pRfg-ortho} $f_{g}^{(l)}f_{h}^{(l)} = 0$ for $l=0,1,\ldots,d$ and $g,h \in G$ with $g \neq h$.
\item $\left\| \displaystyle \sum_{l=0}^{d}\sum_{g \in G}  f_{g}^{(l)} - 1
\right\|<\eps$.
\item $\left\|\big[f_{g}^{(l)},a \big]\right\| < \eps$ for all $l \in \{ 0,1
,\ldots, d\}$, $g\in G$, and $a \in F$.
\item \label{def-pRfg-permut} $\left\|\alpha_h \Big(f_{g}^{(l)}\Big) - f_{hg}^{(l)} \right\| <\eps$ for
all $l \in \{ 0 ,1,\ldots, d\}$ and 
$g \in G$.
\item $\left\| \left [ f_{g}^{(l)},f_{h}^{(k)} \right ]\right\|< \eps$ for any
$k,l \in \{ 0 ,1,\ldots, d\}$ and for any $g,h\in G$.
\end{enumerate} 
\end{Def}

The definition is equivalent if we replace the orthogonality condition (\ref{def-pRfg-ortho}) above by the formally weaker
condition: $\left\|f_{g}^{(l)}f_{h}^{(l)}\right\|< \eps$. If we weaken this condition
in this way, then we can strengthen the group translation condition (\ref{def-pRfg-permut}) to be exact:
$\alpha_h \Big(f_{g}^{(l)}\Big) = f_{hg}^{(l)}$.
 
The following equivalent formulation is a straightforward exercise, and the
proof will be omitted. We define $\A_{\infty} = l^{\infty}(\N,\A)/c_0(\N,\A)$,
with $\A$ identified with the subalgebra of constant sequences in $\A_{\infty}$.
We denote by $\overline{\alpha}$ the induced actions of $G$ on $\A_{\infty}$ and on
$\A_{\infty} \cap \A'$. (We caution the reader that there are conflicting conventions concerning notation for sequence algebras in the literature; some authors use $\A^{\infty}$ for what we call $\A_{\infty}$, and $\A_{\infty}$ for what we call $\A_{\infty} \cap \A'$.)

\begin{Lemma}
\label{Lemma:unital-central-sequence-reformulation-finite-group}
Let $G$ be a finite group, let $\A$ be a unital separable $C^*$-algebra, and let $\alpha \colon G
\to \textup{Aut} (\A)$ be an action of $G$ on $\A$.
Then $\dimrokct(\alpha)=d$ if and only if $d$ is the least integer such that the
following holds:  there is a family  
$\left( f^{(l)}_g \right)_{l=0,1,\ldots,d \, ; g \in G}$ of positive contractions in  $\A_{\infty} \cap
\A'$
such that
\begin{enumerate}
\item $f_{g}^{(l)}f_{h}^{(l)} = 0$, for $l \in \{ 0,1
,\ldots, d\}$ and $g,h \in G$ with $g \neq h$.
\item $\displaystyle \sum_{l=0}^{d}\sum_{g \in G}  f_{g}^{(l)} =1$.
\item $\overline{\alpha}_h \Big(f_{g}^{(l)}\Big) = f_{hg}^{(l)}$ for all $l \in \{ 0,1
,\ldots, d\}$ and 
$g \in G$.
\item $\left [ f_{g}^{(l)},f_{h}^{(k)} \right ] =0$ for any $k,l \in \{ 0,1
,\ldots, d\}$ and for any $g,h\in G$.
\end{enumerate} 
\end{Lemma} 
 
\begin{Def}\label{Def:ecm-map}
Let $G$ be a compact group,
let $\A$ and $\Ch$ be unital $C^*$-algebras,
and let $\alpha \colon G \to \aut (\A)$
and $\gamma \colon G \to \aut (\Ch)$
be actions of $G$ on $\A$ and $\Ch$. Let $F_0 \subseteq \Ch$ and $F \subseteq
\A$ be finite sets, and let $\eps>0$. A unital completely positive map $Q\colon \Ch \to \A$ is said to
be an \emph{$(F_0,F,\eps)$-equivariant central multiplicative map} if:
\begin{enumerate}
\item
$\|Q (x y) - Q (x) Q (y) \|<\eps$ 
for all $x, y \in F_0$.
\item
$\|Q (x) a - a Q (x) \| < \eps$
for all $x \in F_0$ and all $a \in F.$
\item
\label{ecm-map-almost-equivariance}
$ \sup_{g \in G} \| Q (\gamma_g (x)) - \alpha_g (Q (x)) \| < \eps$
for all $x \in F_0$.
\end{enumerate}
If for any such $F_0,F,\eps$ there is an $(F_0,F,\eps)$-equivariant central
multiplicative map from $\Ch$ to $\A$ then we say that $\A$ admits an
\emph{approximate equivariant central unital homomorphism from $\Ch$.}
\end{Def}

\begin{Rmk}
\label{Rmk:ecm-map-equiv}
One can replace condition (\ref{ecm-map-almost-equivariance}) in Definition \ref{Def:ecm-map} with the requirement that the map $Q$ be equivariant. To see that, we first notice that we can require instead that $F$ and $F_0$ be compact and get an equivalent definition. Fix $F_0$, $F$, and $\eps$ as in Definition \ref{Def:ecm-map}. Assume without loss of generality that all elements of $F$ and $F_0$ have norm at most $1$. Fix a map $Q$ as in Definition \ref{Def:ecm-map}, where $F$ and $F_0$ are replaced by their orbits under $G$, and $\eps$ is replaced by $\eps/2$. Define
$$
\widetilde{Q}(x) = \int_G \alpha_{g^{-1}} (Q(\gamma_g(x))) dg \; .
$$ 
It is easy to see that $\widetilde{Q}$ is a $G$-equivariant map that satisfies the conditions of Definition \ref{Def:ecm-map}.
\end{Rmk}

We can reformulate this property in terms of the central sequence algebra as
well, under some stricter assumptions.
\begin{Lemma}\label{Lemma:ecm-central-sequence}	
Let $G$ be a finite group,
let $\A$ and $\Ch$ be unital separable $C^*$-algebras, with $\Ch$ nuclear, 
and let $\alpha \colon G \to \aut (\A)$
and $\gamma \colon G \to \aut (\Ch)$
be actions of $G$ on $\A$ and $\Ch$. Then $\A$ admits an approximate equivariant
central unital homomorphism from $\Ch$ if and only if there is an equivariant
unital homomorphism $\Psi\colon\Ch \to \A_{\infty} \cap A'$.
\end{Lemma}
\begin{proof}
Suppose $\A$ admits an approximate equivariant central unital homomorphism from
$\Ch$. Since $\A$ and $\Ch$ are separable, we can choose increasing sequences of finite sets  $F_0(n)\subseteq \Ch$ and $F(n) \subseteq \A$ such that $\bigcup_n F_0(n)$ is dense in $\Ch$ and $\bigcup_n F(n)$ is dense in $\A$. We now choose a
sequence of $(F_0(n),F(n),2^{-n})$-equivariant central multiplicative maps $Q_n\colon
\Ch \to \A$. We define $\Psi$ to be the composition of the map $(Q_1,Q_2,\ldots)\colon\Ch \to l^{\infty}(\A)$ with the quotient onto $\A_{\infty}$. 

Conversely, if $\Psi\colon\Ch \to \A_{\infty} \cap
\A'$ is a homomorphism as in the statement, we find a unital completely positive lifting $Q =
(Q_1,Q_2,\ldots)\colon\Ch \to l^{\infty}(\A)$ using the Choi-Effros lifting theorem. It is readily verified that for any finite subsets $F_0 \subseteq \Ch$ and $F\subseteq \A$ and for any $\eps>0$, $Q_n$ will be an
$(F_0,F,\eps)$-equivariant central multiplicative map for all sufficiently large
$n$.
\end{proof}

We now introduce the following further generalization of Rokhlin dimension.
\begin{Def}\label{D_3410_XRokh}
Let $G$ be a compact group,
let $\A$ be a unital $C^*$-algebra,
and let $\alpha \colon G \to \aut (\A)$ be an action of $G$ on $\A.$
Let $X$ be a compact free $G$-space.
We say that $\alpha$ has the
{\emph{$X$-Rokhlin property}} if
$\A$ admits an approximate equivariant central unital homomorphism from
 $C (X).$
\end{Def}

The following lemma shows that this is indeed a generalization of finite Rokhlin
dimension.
\begin{Lemma}\label{L_3410_CommTw}
For every finite group~$G$
and every nonnegative integer $d$
there is a compact metrizable free $G$-space~$X$
such that an action $\alpha \colon G \to \aut (\A)$
on a unital $C^*$-algebra $A$ has $\dimrokct(\alpha) \leq d$
if and only if $\alpha$ has the $X$-Rokhlin property.
\end{Lemma}
\begin{proof}
Consider the universal $C^*$-algebra $\Ch$ generated by a family $\left ( f_g^{(k)}\right )_{g\in G, k=0,1,\ldots,d}$ of commuting positive
contractions satisfying $\displaystyle
\sum_{g,k}f_g^{(k)} = 1$ and $f_g^{(k)}f_h^{(k)}=0$ whenever $g \neq h$. 
It admits an action $\gamma$ of $G$, determined by
$\gamma_g\big (f_h^{(k)} \big ) = f_{gh}^{(k)}$. We now take $X$ to be the Gelfand spectrum
of this $C^*$-algebra, which can be identified with a compact subset of the cube
$[0,1]^{\# G\cdot (d+1)}$. (In fact one can check that it is a finite cell
complex, but we make no use of this fact in this paper.) We claim that the action of $G$
on $X$, which we also call $\gamma$, must be free. To see that, we view the elements
$f_g^{(k)}$ as functions on $X$. Let $x \in X$. Pick $g,k\in G$ such that
$f_g^{(k)}(x)>0$. But if $h \in G \smallsetminus \{1\}$ then $f_g^{(k)}(\gamma_h(x)) =
f_{h^{-1}g}^{(k)}(x) = 0$ since $f_g^{(k)}f_{h^{-1}g}^{(k)}=0$. The claim is proved.
The statement of the lemma now follows immediately from Lemmas
\ref{Lemma:unital-central-sequence-reformulation-finite-group} and
\ref{Lemma:ecm-central-sequence}.
\end{proof}

\begin{Rmk}
The space $X$ can be computed explicitly, although we do not use it in this
paper. For example, if $G = \Z_2$, one can show that $X \cong S^d$, with the
action given by multiplication by $-1$. We omit the details.
\end{Rmk}

\begin{Lemma}
\label{Lemma:X-Rokhlin-implies-finite-Rokhlin-dim}
Let $\A$ be a unital separable $C^*$-algebra, let $G$ be a finite group, and let
$\alpha\colon G \to \aut(\A)$ be an action. Let $X$ be a compact free $G$-space with
covering dimension at most $d$. If $\alpha$ has the $X$-Rokhlin property then
$\dimrokct(\alpha) \leq d$.
\end{Lemma}

In order to prove the lemma, we recall the characterization of covering dimension in terms of decomposable covers. 

\begin{Def}
Let $X$ be a set. A family of subsets $(U_j)_{j \in I}$ is said to be \emph{$d$-decomposable} if there is a decomposition $I = \coprod_{k=0}^d I_k$ such that for any $k = 0,1,\ldots,d$ and any $j,j' \in I_k$, if $j \neq j'$ then $U_j \cap U_{j'} = \varnothing$.
\end{Def} 

\begin{Prop}{\cite[Proposition 1.5]{kirchberg-winter}}
Let $X$ be a normal topological space. The space $X$ has covering dimension at most $d$ if and only if every finite open cover of $X$ has a $d$-decomposable finite open refinement.
\end{Prop}

\begin{proof}[Proof of Lemma \ref{Lemma:X-Rokhlin-implies-finite-Rokhlin-dim}]
The quotient map $\pi\colon X \to
X/G$ is a local homeomorphism. Since the space $X/G$ is the image of $X$ under a local homeomorphism, it also has covering dimension at most $d$. (The space $X/G$ can be written as a union of finitely closed subsets, each of which is homeomorphic to a closed subspace of $X$, and thus its dimension is bounded above by the dimension of $X$ by \cite[Corollary 50.3]{Munkres-textbook}.) 
   Pick a finite open cover $(U_j)_{j=1,2,\ldots,n}$ of $X/G$ such that for any $j$, $\pi^{-1}(U_j)$ is homeomorphic to
$\# G$ disjoint copies of $U_j$, that is, for any $j$ there is an open subset $W_j \subseteq X$ such that 
$\pi|_{W_j}: W_j \to U_j$ is a homeomorphism and $\pi^{-1}(U_j) = \coprod_{g \in G} W_j\cdot g$.

 By passing to an open refinement, we may assume without loss of generality that the
cover $(U_j)_{j=1,2,\ldots,n}$ is $d$-decomposable. Pick a partition of unity
$(h_j)_{j=1,\ldots,n}$ of $X/G$ such that $\supp(h_j) \subseteq U_j$ for all $j$.

Since the cover $(U_j)_{j=1,2,\ldots,n}$
is $d$-decomposable, we can partition $\{1,2,\ldots,n\}$
into $d+1$ subsets $I_0,I_1,\ldots,I_d$ such that for any $k$ and any $j,j' \in I_k$, if  $j \neq j'$ then $U_j \cap U_{j'} = \varnothing$. In particular, for any $k$ and any $j,j' \in I_k$, if  $j \neq j'$ then $h_{j}h_{j'} = 0$. 

For $j=1,2,\ldots,n$ define $\widetilde{h}_j \in C(X)$ by 
$$
\widetilde{h}_j(x) = \left \{ 
\begin{matrix} 
h_j(\pi(x)) & \mid & x \in W_j \\
0 & \mid & \mathrm{otherwise} \, .
\end{matrix}
\right  .
$$
(It is easy to check that $\widetilde{h}_j$ is indeed continuous.)

Clearly, for any $k$ and any $j,j' \in I_k$, if $j \neq j'$ then $\widetilde{h}_j \widetilde{h}_{j'} = 0$.
  Define $f_1^{(k)}
= \sum_{j \in I_k} \widetilde{h}_j$. Denoting by $\gamma$ the action of $G$ on
$C(X)$, we now define $f_g^{(k)} = \gamma_g(f_1^{(k)})$. By our construction,
$f_g^{(k)}f_h^{(k)} = 0$ if $g \neq h$ and these functions form a partition of unity of
$X$. 

Let $\Psi\colon C(X) \to \A_{\infty} \cap \A'$ be an equivariant unital homomorphism. 
The elements $\Psi(f_g^{(k)})$ satisfy the conditions of Lemma
\ref{Lemma:unital-central-sequence-reformulation-finite-group}. Thus, $\dimrokct(\alpha) \leq d$ as required.
\end{proof}
 
\begin{Exl}	
\label{Exl:irrational-rotation-algebra}
Let $\theta \in (0,1)$ be an irrational number, and let $\A_{\theta}$ be the
irrational rotation algebra. Let $u$ and $v$ be the canonical unitary generators of
$\A_{\theta}$, satisfying $uv = e^{2\pi i \theta} vu$. Let $\alpha$ be the order
2 automorphism of $\A_{\theta}$ given by $\alpha(v)=v$ and $\alpha(u) = -u$, and
think of $\alpha$ as defining an action of $\Z_2$ on $\A_{\theta}$. We claim
that $\dimrokct(\alpha) = 1$.

To see that, let $(n_k)_{k \in \N}$ be a sequence of odd integers satisfying $\displaystyle
\lim_{k \to \infty}\textrm{dist} (n_k\theta,\Z) = 0$. Since
$vu^{n_k} = e^{2\pi i n_k \theta} uv$, one sees that for any $a \in \A_{\theta}$
we have  $\displaystyle \lim_{k \to \infty}\|[a,u^{n_k}]\| = 0$.
Since $n_k$ is odd, $\alpha(u^{n_k}) = - u^{n_k}$. Identifying the unitaries
$u^{n_k}$ with equivariant unital homomorphisms from $C(\T)$ to $\A_{\theta}$
(where the action on $\T$ is rotation by $\pi$), we see that $\dimrokct(\alpha)
\leq 1$ by Lemma \ref{Lemma:X-Rokhlin-implies-finite-Rokhlin-dim}.

However, $\alpha$ does not have the Rokhlin property (that is,
$\dimrokct(\alpha) \neq 0$). In fact, no action of any nontrivial finite group on $\A_{\theta}$ has the Rokhlin property. To see that, observe that any automorphism induces that identity map on $K_0$. Thus, if $\alpha\colon  G\to \aut(\A_{\theta})$ had the Rokhlin property, then there would be a family of projections $(p_g)_{g \in G}$ in $\A_{\theta}$, all of which have the same $K_0$ class, such that $\sum_{g \in G} p_g = 1$. Therefore, $[1]$ would be divisible by $\# G$, which is false. (For a more elaborate discussion of obstructions to the Rokhlin property, see \cite[Proposition 3.13]{Phillips-survey} and the surrounding discussion.)

A similar argument shows that the action of $\Z_p$ on $\A_{\theta}$ which fixes
$v$ and sends $u$ to $e^{2\pi i /p}u$ has Rokhlin dimension 1 with commuting towers. (An argument of a similar nature is used to show that certain actions of $\R$ on $\A_{\theta}$ have the Rokhlin property. See \cite[Proposition 2.5]{Kishimoto-flows}.)
\end{Exl}

We record the following straightforward lemma, without proof, for further use. 
\begin{Lemma}\label{L_3410_TPrdCt}
Let $G$ be a compact group,
let $\A,$ $\B,$ and $\Ch$ be unital $C^*$-algebras,
and let $\alpha \colon G \to \aut (\A),$
$\beta \colon G \to \aut (\B),$
and $\gamma \colon G \to \aut (\Ch)$
be actions of $G$ on $\A,$ $\B,$ and $\Ch.$
Suppose that $\A$  admits an approximate equivariant central unital homomorphism
from
 $\Ch.$
Then, for any $C^*$-tensor product for which the
diagonal action
$g \mapsto \alpha_g \otimes \beta_g$ of $G$ on $\A \otimes \B$
is defined,
$\A \otimes \B$
 admits an approximate equivariant central unital homomorphism from
$\Ch.$
\end{Lemma}

We now extend the definition of finite Rokhlin dimension for actions of finite
groups and of a single automorphism to the nonunital case. This definition will
be sufficient for extending the permanence properties from \cite{HWZ} to the
nonunital setting. We begin with the finite group case.

\begin{Def}
\label{def: nonunital Rokhlin finite groups}
Let $G$ be a finite group, let $\A$ a $C^*$-algebra and let $\alpha \colon  G \to \textup{Aut}
(\A)$ an action of $G$ on $\A$.
We say that $\alpha$ has \emph{Rokhlin dimension $d$ with commuting towers}, and
write $\dimrokct(\alpha)=d$, if $d$ is the least integer such that the following
holds: for any 
$\eps >0$ and every finite subset $F \subset \A$ there is a family $\left( f^{(l)}_g \right)_{l=0,1,\ldots,d\, ; g \in G}$ of positive contractions in $\A$
such that:
\begin{enumerate}
\item 
\label{def-finite-group-ortho}
$\left\|f_{g}^{(l)}f_{h}^{(l)}a\right\|< \eps$ for  $l=0,1,\ldots,d$, any $a \in \A$, and any $g, h$ in $G$ with $g \neq h$.
\item 
\label{def-finite-group-unit}
$\left\| \left ( \displaystyle \sum_{l=0}^{d}\sum_{g \in G} 
f_{g}^{(l)}\right )a - a \right\|<\eps$ for $a\in F$.
\item 
\label{def-finite-group-central}
$\left\|\big[f_{g}^{(l)},a \big]\right\| < \eps$ for $l \in \{ 0,1
,\ldots, d\}$, $g\in G$, and $a \in F$.
\item 
\label{def-finite-group-permuted}
$\left\|\left ( \alpha_h \Big(f_{g}^{(l)}\Big) - f_{hg}^{(l)} \right ) a \right\| <\eps$ for
 $l \in \{ 0 ,1,\ldots, d\}$, $a \in F$, and 
$g,h \in G$.
\item
\label{commuting tower assumption}
 $\left\| \left [ f_{g}^{(l)},f_{h}^{(k)} \right ]a\right\|< \eps$ for $k,l
\in \{ 0 ,1,\ldots, d\}$, $a \in F$, and $g,h\in G$.
\end{enumerate} 
\end{Def}
\begin{Def}
In the notation of Definition \ref{def: nonunital Rokhlin finite groups}, given
$F \subseteq \A$ finite and $\eps>0$, we call a family $\left(
f^{(l)}_g \right)_{l=0,1,\ldots,d\, ; g \in G}$  of positive elements satisfying the conditions of
Definition \ref{def: nonunital Rokhlin finite groups} with respect to the given
$F$ and $\eps$ a $(d,F,\eps)$-\emph{Rokhlin system}. 
\end{Def}

As in the unital case, we have the following equivalent reformulation using the
central sequence algebra.

\begin{Lemma}
\label{Lemma:central-sequence-reformulation-finite-group}
Let $G$ be a finite group, let $\A$ be a separable $C^*$-algebra, and let $\alpha \colon  G \to
\textup{Aut} (\A)$ be an action of $G$ on $\A$.
Then $\dimrokct(\alpha)=d$ if and only if $d$ is the least integer such that the
following holds:  there is a family 
$\left( f^{(l)}_g \right)_{l=0,1,\ldots,d\, ; g \in G}$ of positive contractions  in $\A_{\infty} \cap
\A'$
such that
\begin{enumerate}
\item 
\label{central-sequence-finite-group-ortho}
$f_{g}^{(l)}f_{h}^{(l)} a = 0$ for  $l=0,1,\ldots,d$, any $a \in \A$, and any $g, h \in G$ with $g \neq h$.
\item $\left ( \displaystyle \sum_{l=0}^{d}\sum_{g \in G}  f_{g}^{(l)}\right )a
= a$ for all $a\in \A$.
\item $\overline{\alpha}_h \Big(f_{g}^{(l)}\Big)a = f_{hg}^{(l)}a$ for all $l \in \{ 0,1
,\ldots, d\}$, $a \in \A$, and 
$g \in G$.
\item $\left [ f_{g}^{(l)},f_{h}^{(k)} \right ]a =0$ for any $k,l \in \{ 0,1
,\ldots, d\}$, any $a \in \A$, and any $g,h\in G$.
\end{enumerate} 
\end{Lemma}

\begin{Rmk}
\label{Rmk:central-sequence-reformulation-finite-group-Moreover}
 With the notation of Lemma \ref{Lemma:central-sequence-reformulation-finite-group} above, if $\dimrokct(\alpha)=d$ and $\B \subseteq \A_{\infty}$ is any separable subset, then the family $\left( f^{(l)}_g \right)_{l=0,1,\ldots,d\, ; g \in G}$ can in addition be chosen to satisfy
$$
f^{(l)}_gb = b f^{(l)}_g
$$
for all $b \in \B$, for $l=0,1,\ldots,d$, and for all $g \in G$.

This is shown using a standard diagonalization method.
\end{Rmk}

\begin{Rmk}
\label{Rmk:finite-group-ortho}
Condition (\ref{central-sequence-finite-group-ortho}) in Lemma \ref{Lemma:central-sequence-reformulation-finite-group} can be strengthened to require that $f_g^{(l)}f_h^{(l)} = 0$, rather than obtaining $0$ only after multiplying by an element of $\A$. To see this, let $\left ( f_g^{(l)} \right )_{l=0,1,\ldots,d \, ;  g \in G}$ be a system in $\A_{\infty} \cap \A'$ as in the lemma. The annihilator $\mathrm{Ann}(\A)$ is an ideal in $\A_{\infty} \cap \A'$. Let $\pi \colon \A_{\infty} \cap \A' \to \A_{\infty} \cap \A' / \mathrm{Ann}(\A)$ be the quotient map. Any system of contractive lifts of $\left ( \pi(f_g^{(l)}) \right )_{l=0,1,\ldots,d \, ;  g \in G}$ to $\A_{\infty} \cap \A'$ will satisfy the conditions of Lemma \ref{Lemma:central-sequence-reformulation-finite-group} as well. Since orthogonal contractions can be lifted to orthogonal contractions (the cone over $\C^n$ is projective), we can choose Rokhlin elements with this added orthogonality condition. Likewise, one shows that strengthening condition (\ref{def-finite-group-ortho}) in Definition \ref{def: nonunital Rokhlin finite groups} to require that $f_g^{(l)}f_h^{(l)} = 0$ gives an equivalent definition.
\end{Rmk}

We record the following simple observation. The proof is immediate.
\begin{Lemma}
\label{Lemma:finite-subgroup}
Let $G$ be a finite group, let $\A$ be a $C^*$-algebra, and let $\alpha \colon G \to \aut(\A)$ be an action with $\dimrokct(\alpha)\leq d$. For any subgroup $H < G$ we have $\dimrokct(\alpha|_{H}) \leq d$.
\end{Lemma}

\begin{Lemma}
\label{Lemma:Rokhlin-pointwise-outer}
Let $G$ be a finite group, let $\A$ be a $C^*$-algebra, and let $\alpha \colon G \to \aut(\A)$ be an action with $\dimrokct(\alpha)<\infty$. Then the action $\alpha$ is pointwise outer.
\end{Lemma}
\begin{proof}
 Suppose not. Let $h$ be a nontrivial element of $G$, and suppose that there is a unitary $u \in M(\A)$ such that $\alpha_h(a) = uau^*$ for all $a \in \A$. Let $a$ be a nonzero $G$-invariant positive element in $\A$ of norm $1$. Let $d=\dimrokct(\alpha)$. Fix $\eps>0$ such that
$$
\left ( \frac{1-\eps}{(d+1) \cdot \# G} \right )^2 > 5\eps \; .
$$
 Let $(f_g^{(l)})_{l=0,1,\ldots,d \, ; g \in G}$ be a Rokhlin system for the finite set $\{a^{1/2},u^*a^{1/2},a^{1/2}u\}$ and $\eps$. Since 
$$
\left \|\left ( \sum_{g \in G} \sum_{l=0}^d f_g^{(l)} \right ) a - a \right \| < \eps \, ,
$$
there exist $g \in G$ and $l \in \{0,1,\ldots d\}$ such that $\|af_g^{(l)}\| > \frac{1-\eps}{(d+1) \cdot \# G}$.
We have 
$$
\left \|a^{1/2}uf_g^{(l)}u^*a^{1/2} - af_g^{(l)} \right \|<2\eps
\quad {\mbox{and}}\quad 
\left \| a^{1/2}uf_g^{(l)}u^*a^{1/2} - af_{hg}^{(l)} \right \|<2\eps
\, .
$$
Thus 
$$
\left ( \frac{1-\eps}{(d+1) \cdot \# G} \right )^2 <
\left \|(af_g^{(l)})(af_g^{(l)})^* \right \| \leq 
\left \| af_g^{(l)}f_{hg}^{(l)}a \right \| + 4\eps \leq
5\eps \, ,
$$
which is a contradiction.
\end{proof}

Now we consider the case of a single automorphism.

\begin{Def} \label{nonunital Rokhlin-dim-single-auto}
Let $\A$ be a $C^*$-algebra and $d \in \N$. An automorphism $\alpha$ of $\A$ is
said to have \emph{Rokhlin dimension  
 $d$ with commuting towers} if $d$ is the least integer such that the following
holds: for any finite set $F \subset \A$, 
any $p>0$, and any $\eps>0$, there are positive elements 
$$
f_{0,0}^{(l)}\, , f_{0,1}^{(l)}\, ,\ldots,f_{0,\, p-1}^{(l)}\;\; \mathrm{and} \;\; f_{1,0}^{(l)} \, , f_{1,1}^{(l)} \, , \ldots,f_{1,\, p}^{(l)}
$$ 
for $l =
0,1,\ldots,d$ such that:
\begin{enumerate}
\item 
\label{def-single-auto-ortho}
$\|f_{q,k}^{(l)}f_{r,j}^{(l)}a\|<\eps$ for any $a \in F$, $l=0,1,\ldots,d$, for $q,r = 0,1$, for  $k = 0,1,\ldots,p-1+q$ and $j=0,1,\ldots,p-1+r$ with $(q,k) \neq
(r,j)$.
\item $\left\|\left ( \displaystyle \sum_{l=0}^{d} \left [
\sum_{j=0}^{p-1}
f_{0,j}^{(l)} + \sum_{j=0}^{p}
f_{1,j}^{(l)} 
\right ]
\right ) a - a \right\|<\eps$ for all $a \in F$.
\item $\left\|[f_{r,j}^{(l)},a]\right\| < \eps$ for  $l=0,1,\ldots,d$, for $r=0,1$,  $j=0,1,\ldots,p-1+r$ and for $a \in F$.
\item $\left\|\left ( \alpha(f_{r,j}^{(l)}) - f_{r,j+1}^{(l)} \right ) a\right\|<\eps$ for 
$l=0,1,\ldots,d$, for $r = 0,1$, for $j=0,1,\ldots,p-2+r$ and for all $a \in F$.
\item  $\left\|\left ( \alpha(f_{0,p-1}^{(l)} + f_{1,p}^{(l)}) - (f_{0,0}^{l} +
f_{1,0}^{l}) \right ) a\right\|<\eps$ for  $l=0,1,\ldots,d$ and for all $a \in F$.
\item $\|[f_{q,k}^{(l)},f_{r,j}^{(m)}]a\|< \eps$ for all $a\in F$, for $l,m = 0,1,\ldots,d$, for $q,r = 0,1$,  for $k = 0,1,\ldots,p-1+q$ and for $j=0,1,\ldots,p-1+r$.
\end{enumerate} 
We write in this case $\dimrokct(\alpha)=d$.

We refer to each sequence
$f_{r,0}^{(l)},f_{r,1}^{(l)},f_{r,2}^{(l)},\ldots$ or to $\left(f_{r,j}^{(l)}
\right)_{j=0,1,\ldots,p-1+r}$ as a \emph{tower}, to the length of the sequence as the \emph{height}
of the tower, and to the 
pair of towers for $r=0,1$ as a double tower. If the double tower satisfies the
conditions with respect to a given $(d,F,\eps)$, we  refer to those
elements as a \emph{$(d,F,\eps)$-double tower of height $p$}.
\end{Def}

\begin{Exl} Rokhlin dimension zero for automorphisms of nonunital $C^*$-algebras coincides with the definition of the Rokhlin property for nonunital $C^*$-algebras from \cite[Definition 1.2]{brown-hirshberg}). (Formally, the definition of the Rokhlin property in \cite{brown-hirshberg} is slightly stronger: it is reformulated as in Lemma \ref{Lemma:central-sequence-reformulation-single-auto}, except that instead of item (\ref{Lemma:central-sequence-reformulation-single-auto:orthogonality}), the elements in question are required to be orthogonal even without multiplying by an element from $\A$; however, those definitions are equivalent --- see Remark \ref{Rmk:single-auto-ortho} below.)

Such automorphisms can arise from endomorphisms which satisfy the Rokhlin property. To give a concrete example, we review the representation of $\Oh_n$ as a
corner in a crossed product, from \cite[Section 2]{Cuntz}. Consider $M_{n^{\infty}} \cong
M_n \otimes M_n \otimes \cdots$. Let $e \in M_n$ be a fixed minimal projection.
Let $\alpha\colon M_{n^{\infty}} \to M_{n^{\infty}}$ be the nonunital endomorphism
given by $\alpha(a_1 \otimes a_2 \otimes \cdots) = e \otimes a_1 \otimes a_2
\otimes \cdots$. Let $\widetilde{\alpha}$ be the induced automorphism on the
stationary inductive limit $\K \otimes M_{n^{\infty}} \cong
\underset{\longrightarrow}{\lim}(M_{n^{\infty}},\alpha)$. One can check that
$\widetilde{\alpha}$ has Rokhlin dimension $0$ (see \cite[Proposition 2.2]{brown-hirshberg}). It follows then from Theorem
\ref{Thm:permanence-dimension-single-auto} below that the crossed product, which is isomorphic to $\Oh_n
\otimes \K$, has finite nuclear dimension. The bound for nuclear dimension given in the statement
of the theorem is 3. However when $n$ is even, one can obtain single Rokhlin towers of height
$2^k$ (see \cite[Proposition 4.1 and Remark 4.3]{BSKR}), and therefore the proof of the theorem can in fact be used to yield
nuclear dimension 1. Since finite nuclear dimension passes to hereditary
subalgebras, the same holds for $\Oh_n$ as well. That $\Oh_n$ has finite nuclear
dimension was shown in \cite[Theorem 7.4]{winter-zacharias} using a
different argument not involving the Rokhlin property, from which it follows (\cite[Theorem 7.5]{winter-zacharias}) that the same holds for general Kirchberg algebras satisfying the UCT. (The bound on the nuclear dimension for such algebras was improved recently; see \cite{enders,rss-nuc-dim}.) 
The construction of Kirchberg algebras as corners of crossed products of AF algebras by automorphisms with Rokhlin dimension 0 can be carried out in greater generality. It was shown in \cite[Theorem 3.1 and Corollary 4.6]{rordam95} that for any pair of abelian groups $G_0$ and $G_1$ with $G_1$ torsion free and any $g_0 \in G_0$ one can obtain in this way a Kirchberg algebra $A$ with $(K_0(A),[1],K_1(A)) \cong (G_0,g_0,G_1)$.
\end{Exl}

As in the case of a finite group action, we can reformulate Rokhlin dimension for a single automorphism in terms of the central sequence algebra.

\begin{Lemma}
\label{Lemma:central-sequence-reformulation-single-auto}
Let $\A$ a separable $C^*$-algebra and let $\alpha \in \aut (\A)$.
Then $\dimrokct(\alpha)=d$ if and only if $d$ is the least integer such that the
following holds:  for any integer $p>0$ there are positive contractions 
$$
f_{0,0}^{(l)},f_{0,1}^{(l)},\ldots,f_{0,p-1}^{(l)}, f_{1,0}^{(l)}, f_{1,1}^{(l)},\ldots,f_{1,p}^{(l)}
$$ 
for $l = 0,1,\ldots,d$
in $ \A_{\infty} \cap \A'$
such that:
\begin{enumerate}
\item 
\label{Lemma:central-sequence-reformulation-single-auto:orthogonality}
$f_{q,k}^{(l)}f_{r,j}^{(l)}a=0$ for any $a \in \A$, for $l=0,1,\ldots,d$, for  $q,r = 0,1$,  for $k = 0,1,\ldots,p-1+q$ and $j=0,1,\ldots,p-1+r$ with $(q,k) \neq
(r,j)$.
\item $\left ( \displaystyle \sum_{l=0}^{d} \left [
\sum_{j=0}^{p-1}
f_{0,j}^{(l)} + \sum_{j=0}^{p}
f_{1,j}^{(l)} 
\right ]
\right ) a = a$ for all $a \in \A$.
\item $\left ( \overline{\alpha}(f_{r,j}^{(l)}) - f_{r,j+1}^{(l)} \right ) a=0$ for 
$l=0,1,\ldots,d$, for $r = 0,1$, for $j=0,1,\ldots,p-2+r$, and for all $a \in \A$.
\item  $\left ( \overline{\alpha}(f_{0,p-1}^{(l)} + f_{1,p}^{(l)}) - (f_{0,0}^{l} +
f_{1,0}^{l}) \right ) a=0$ for  $l=0,1,\ldots,d$ and for all $a \in \A$.
\item $[f_{q,k}^{(l)},f_{r,j}^{(m)}]a=0$ for  $a\in \A$, for $l,m = 0,1,\ldots,d$, for $q,r = 0,1$,  for $k = 0,1,\ldots,p-1+q$, and for $j=0,1,\ldots,p-1+r$.
\end{enumerate} 
\end{Lemma}

\begin{Rmk}
\label{Rmk:central-sequence-reformulation-single-auto-Moreover}
As in the case of finite group actions, with the notation of Lemma \ref{Lemma:central-sequence-reformulation-single-auto} above, if $\dimrokct(\alpha)=d$ and $\B \subseteq \A_{\infty}$ is any separable subset, then the Rokhlin elements $f_{0,0}^{(l)},f_{0,1}^{(l)},\ldots,f_{0,p-1}^{(l)}, f_{1,0}^{(l)}, f_{1,1}^{(l)},\ldots,f_{1,p}^{(l)}
$ 
can in addition be chosen to satisfy
$$
f^{(l)}_{r,j}b = b f^{(l)}_{r,j}
$$
for all $b \in \B$, for all $l=0,1,\ldots,d$, for $r=0,1$, and for $j=0,1,\ldots,p-1+r$.
\end{Rmk}

\begin{Rmk}
\label{Rmk:single-auto-ortho}
Condition (\ref{Lemma:central-sequence-reformulation-single-auto:orthogonality}) in Lemma \ref{Lemma:central-sequence-reformulation-single-auto} can be strengthened to require that
$f_{q,k}^{(l)}f_{r,j}^{(l)}=0$, for the same indices that appear there. The proof of this is the same as in Remark \ref{Rmk:finite-group-ortho}. Likewise, one can strengthen condition (\ref{def-single-auto-ortho}) in Definition \ref{nonunital Rokhlin-dim-single-auto} to require that $f_{q,k}^{(l)}f_{r,j}^{(l)}=0$.
\end{Rmk}

The paper \cite{Kirchberg-Abel} is devoted to a study of the $C^*$-algebra $(\A_{\omega}\cap A')/\mathrm{Ann}(\A)$, where $\omega$ is a free ultrafilter, as a suitable substitute for $\A_{\omega}\cap A'$ when $\A$ is nonunital. We do not use this formalism explicitly here. However it is worth noting that Lemma \ref{Lemma:central-sequence-reformulation-single-auto} takes a rather natural form if one considers the image of the Rokhlin elements in the quotient $(\A_{\infty} \cap \A')/\mathrm{Ann}(\A)$.

\section{Actions of finite groups: permanence properties}
\label{sec:finite-group-permanence-properties}

In this section we consider permanence properties for crossed products by
actions with finite Rokhlin dimension, and study the behavior of such actions
under extensions.

We begin by extending the permanence properties from \cite{HWZ} to the
nonunital setting, which  we state as Theorems
\ref{Thm:permanence-dimension-finite-groups} and
\ref{Thm:permanence-Z-finite-groups}.

\begin{Thm}
\label{Thm:permanence-dimension-finite-groups}
Let $G$ be a finite group, let $\A$ be a $C^*$-algebra with finite decomposition rank
and let $\alpha\colon G\to \textup{Aut}(\A)$
be an action with $\dimrokct(\alpha)=d$. Then the crossed product $\A
\rtimes_{\alpha} G$ has finite 
decomposition rank. In fact, 
$$
\dr(\A \rtimes_{\alpha} G) \leq (\dr(\A) +1) (d+1)-1  \; .
$$ 
The same statement is true 
for nuclear dimension in place  of decomposition rank. 
\end{Thm}

\begin{Thm}
\label{Thm:permanence-Z-finite-groups}
Let $G$ be a finite group, let $\A$ be a separable $\Zh$-absorbing $C^*$-algebra, and
let $\alpha\colon G\to \aut(\A)$
be an action with $\dimrokct(\alpha)<\infty$. Then  $\A \rtimes_{\alpha} G$ is
$\Zh$-absorbing.
\end{Thm}

 Theorem
\ref{Thm:permanence-dimension-finite-groups} is a generalization of \cite[Theorem 1.3]{HWZ} to the nonunital setting. The modification required to obtain this generalization is
straightforward and will be omitted. For Theorem
\ref{Thm:permanence-dimension-finite-groups}, it is not necessary to assume that
the different Rokhlin towers approximately commute (condition (\ref{commuting
tower assumption}) in Definition \ref{def: nonunital Rokhlin finite groups}).

Theorem \ref{Thm:permanence-Z-finite-groups}
requires more argument. We will omit proofs when they are straightforward
modifications or corollaries of results that have appeared elsewhere.

\begin{Lemma}
\label{Lemma:universal-space-commuting-contractions}
Let $X = \{(x_0,x_1) \in [0,1]^2 \mid 0< x_0+x_1 \leq 1\}$.
The universal $C^*$-algebra generated by two commuting positive contractions
$a_0,a_1$ satisfying $a_0+a_1\leq 1$ is isomorphic to $C_0(X)$, in such a way that for $j=0,1$, the element $a_j$ becomes the function   $a_j(x_0,x_1) = x_j$. 
\end{Lemma}
The proof is straightforward and will be omitted. In the above picture, let
$p_0,p_1$ be the support projections of $a_0,a_1$ in $C_0(X)^{**}$, and let $p$ be the
support projection of $a_0+a_1$. It is easy to construct two positive
contractions $g_0,g_1 \in M(C_0(X))$ such that the support projection of $g_0$ is
$p_0$, the support projection of $g_1$ is $p_1$ and $g_0+g_1 = p$. For example, for $r \in (0,1]$ and $\theta \in [0,\pi/2]$ for which $(r\cos(\theta),r\sin(\theta)) \in X$, set $g_1(r\cos(\theta),r\sin(\theta)) =
2\theta/\pi$ and $g_0 = 1-g_1$.

We refer the reader to \cite[Theorem 3.3]{winter-zacharias-order-zero} for 
the structure of order zero maps, which we use below. If $\A,\B$ are $C^*$-algebras with $\A$ unital and $\varphi\colon \A \to \B$ is a completely positive order zero map, then there is a homomorphism $\pi\colon  \A \to M(C^*(\varphi(\A))) \cap \varphi(1)' \subseteq \B^{**}$ such that for all $a \in \A$ we have $\varphi(a) = \pi(a)\varphi(1)$. We call $\pi$ the support homomorphism of $\varphi$. We write 
$$
\Zh_{n,n+1} = \{f \in C([0,1] , M_n \otimes M_{n+1}) \mid f(0) \in M_n \otimes 1 \; \mathrm{and} \; f(1) \in 1 \otimes M_{n+1}\}  
.
$$
One can define order zero contractions $\theta_0 \colon M_n \to \Zh_{n,n+1}$ and $\theta_1 \colon M_{n+1} \to \Zh_{n,n+1}$ by $\theta_0(a)(t) = (1-t)\cdot a\otimes 1_{M_{n+1}}$ and $\theta_1(a)(t) = t\cdot 1_{M_n} \otimes a$. One checks that $\theta_0(1) + \theta_1(1) = 1$, and the images of these two maps generate $\Zh_{n,n+1}$.
\begin{Lemma}
Let $\A$ be a $C^*$-algebra. Suppose $\varphi_0\colon M_{n} \to \A$ and $\varphi_1\colon M_{n+1} \to \A$ are two
contractive order zero maps with commuting images such that
$\varphi_0(1)+\varphi_1(1) \leq 1$. Then there is an order zero map
$\Phi\colon \Zh_{n,n+1} \to \A$ with $\Phi(1) = \varphi_0(1)+\varphi_1(1)$.
\end{Lemma}
\begin{proof}
Write $f_j = \varphi_j(1)$ for $j=0,1$. Let $\pi_0$ and $\pi_1$ be the support homomorphisms of $\varphi_0$ and $\varphi_1$, so that $\varphi_j (a) = \pi_j(a)f_j$. 
 Define $f = f_1+f_2$.
Let $X$ be the space defined in Lemma \ref{Lemma:universal-space-commuting-contractions}, and let $a_0,a_1$ be the two positive elements from that lemma. Then there is a homomorphism $\psi \colon C_0(X) \to \A$ such that $\psi(a_0) = f_0$ and $\psi(a_1) = f_1$. We extend $\psi$ to a homomorphism $\psi^{**} \colon C_0(X)^{**} \to \A^{**}$. Let $g_0$, $g_1$, and $p$ be as in the discussion after Lemma \ref{Lemma:universal-space-commuting-contractions}. 
Set $\widetilde{g}_0 = \psi^{**}(g_0)$, $\widetilde{g}_1 = \psi^{**}(g_1)$, and $\widetilde{p} = \psi^{**}(p)$.
 Define $\widetilde{\varphi}_j \colon M_{n+j} \to \widetilde{p} \A^{**} \widetilde{p}$ for $j=0,1$ by $\widetilde{\varphi}_j(a) = \pi_j(a)g_j$. Then 
$\widetilde{\varphi}_0$ and $\widetilde{\varphi}_1$ are order zero maps from $M_n$ and $M_{n+1}$, respectively, to
$\widetilde{p}\A^{**}\widetilde{p}$ such that $\varphi_0(1)+\varphi_1(1) = \widetilde{p}$. By \cite[Proposition 2.5]{rordam-winter}, they 
therefore give rise to a unital homomorphism $\widetilde{\varphi}\colon \Zh_{n,n+1} \to
\widetilde{p}\A^{**}\widetilde{p}$. Now define $\Phi \colon \Zh_{n,n+1} \to \A^{**}$ by
$\Phi(a) = \widetilde{\varphi}(a)f$ for $a \in \Zh_{n,n+1}$. Then
$\Phi$ is an order zero map with $\Phi(1) = f$, and since $\Zh_{n,n+1}$ is generated by the images of the homomorphisms used in \cite[Proposition 2.5]{rordam-winter}, it is straightforward to verify
that its image is in $\A$.
\end{proof}

\begin{Cor}
\label{cor:d-order-zero-maps}
Let $\A$ be a $C^*$-algebra. Suppose $\varphi_0^{(k)}\colon M_{n} \to \A$, and $\varphi_1^{(k)}\colon M_{n+1} \to \A$, for
$k=0,1,\ldots d$, are contractive order zero maps with commuting images such that 
$$
\sum_{k=0}^d \left [ \varphi_0^{(k)}(1) + \varphi_1^{(k)}(1) \right ] \leq 1 
 .
$$ 
Then there is an order zero map $\Phi\colon \Zh_{n,n+1} \to \A$ with $\Phi(1) =
\displaystyle \sum_{k=0}^d \left [ \varphi_0^{(k)}(1) + \varphi_1^{(k)}(1) \right ] $.

Suppose furthermore $G$ is a discrete group acting on $\A$, and $F \subset
\Zh_{n,n+1}$ is a given finite subset. For any finite subset $G_0 \subseteq
G$ and any $\delta>0$ there exists an $\eps>0$ such that the following holds. If
$\left \|\alpha_g ( \varphi_j^{(k)}(x) ) - \varphi_j^{(k)}(x)\right \| < \eps\|x\|$ for $j = 0,1$, for all $x
\in M_{n+j}$, and for all $g \in G_0$, then there exists an order zero map 
$\Phi$ as above which furthermore satisfies $\|\alpha_g ( \Phi(x) ) - \Phi(x)\|<\delta$ for all $x \in F$
and all $g \in G_0$. 
\end{Cor}
\begin{proof}
Let $D_n^{(m)}$ be the kernel of the canonical map $(CM_n^+)^{\otimes m} \to \C$. It follows by induction from \cite[Lemma 5.2]{HWZ} that $D_n^{(m)}$ satisfies the following universal property with respect to $m$ commuting order zero contractions from $M_n$. Let 
$\eta\colon M_n \to CM_n = C_0((0,1],M_n)$ be the order zero map given by $\eta(a)(t) = ta$. For $j=1,2,\ldots,m$, we define $\eta_j\colon M_n \to D_n^{(m)}$ to be the $j$'th coordinate map 
$$
\eta_j(a) = 1\otimes 1 \otimes \cdots \otimes 1 \otimes \eta(a) \otimes 1 \otimes \cdots \otimes 1 . 
$$
Then if $\A$ is any $C^*$-algebra and $\sigma_1,\sigma_2,\ldots,\sigma_m \colon M_n \to \A$ are contractive order zero maps with commuting images, there exists a (unique) homomorphism $\pi \colon D_n^{(m)} \to \A$ such that $\sigma_j = \pi \circ \eta_j$ for $j=1,2,\ldots,m$. 

By \cite[Lemma 5.3]{HWZ}, if $h$ is any positive element in the center $Z(D_n^{(m)})$, then there exists an order zero map $\theta\colon  M_n \to D_n^{(m)}$ with $\theta(1) = h$ (and with $\|\theta\| = \|h\|$). In particular, it follows that if $\varphi_0^{(0)}, \varphi_0^{(1)},\ldots,\varphi_0^{(d)} \colon M_n \to \A$ are order zero maps as in the statement, then there exists an order zero map $\varphi_0 \colon M_n \to \A$ with $ \varphi_0(1) = \sum_{k=0}^d  \varphi_0^{(k)}(1)$. Similarly, there is an order zero map $\varphi_1 \colon M_{n+1} \to \A \cap \varphi_0(M_n)'$ such that $\varphi_1(1) = \sum_{k=0}^d  \varphi_1^{(k)}(1)$. Therefore, the existence of the map $\Phi$ as in the statement of the corollary follows from the previous lemma. 

The refined statement involving the discrete group action is a modification of the above argument. The added assumption says that $\|\alpha_g(\pi(x)) - \pi(x)\|<\eps$ for all $x$ in a generating set of $D_n^{(m)}$ and for all $g \in G_0$. Therefore, if $\eps$ is chosen to be sufficiently small, we have $\|\alpha_g(\pi(\theta(x))) - \pi(\theta(x))\|<\delta$ for all $x$ in the unit ball of $M_n$. We omit the details.
\end{proof}

We require the following simple adaptation of \cite[Lemma 5.4]{HWZ} to the nonunital
setting, that in turn is based on \cite[Lemma 2.4]{HW}. It uses the
characterization of $\mathcal{D}$-stability from \cite[Proposition 4.1]{HRW}. We recall the notation $\overline{\alpha}$ from after Definition \ref{def: positive Rokhlin finite groups}.

\begin{Lemma}
\label{lemma:enough-to-embed-building-block}
Let $\A$ and $\B$ be separable $C^*$-algebras, with $\B$ unital. Let $G$ be a discrete
countable group with an action $\alpha\colon G \to \aut(\A)$. Suppose that $(\B_n)_{n=1,2,3,\ldots}$ is a
sequence of nuclear subalgebras of $\B$ with dense
union such that $1_{\B} \in \B_n$ for all $n$. Suppose that for any $n \in \N$, any finite subset $F \subseteq \B_n$,  any
$\eps>0$, and any finite set $G_0 \subseteq G$ there is a completely positive contraction $\gamma\colon \B_n
\to \A_{\infty}\cap \A '$ such that:
\begin{enumerate}
\item $\|(\overline{\alpha}_g(\gamma(x)) - \gamma(x))a\|<\eps$ for all $x \in F$, all
$g \in G_0$, and all $a \in \A$ with norm at most 1.
\item $a\gamma(1) = a$ for all $a \in \A$.
\item $a(\gamma(xy) - \gamma(x)\gamma(y)) = 0$ for all $a \in \A$ and $x,y \in
\B_n$.
\label{lemma:enough-to-embed-building-block:multiplicativity-condition}
\end{enumerate}
Then there is a completely positive contraction $\Gamma \colon \B \to \A_{\infty}\cap \A '$ satisfying:
\begin{enumerate}[label=(\arabic*$\, '$)]
\item $a\overline{\alpha}_g(\Gamma(x)) = a\Gamma(x)$ for all $a \in \A$, $x \in \B$, and $g \in G$.
\item $a\Gamma(1) = a$ for all $a \in \A$.
\item $a(\Gamma(xy) - \Gamma(x)\Gamma(y)) = 0$ for all $a \in \A$ and  $x,y \in \B$.
\end{enumerate}

If $\B$ is furthermore strongly self absorbing then the full
crossed product $\A \rtimes_{\alpha} G$ absorbs $\B$ tensorially.
\end{Lemma}
\begin{proof} 
We first claim that the maps $\gamma$ in the hypothesis can be assumed to be defined on all of $\B$. To see that, since $\B_n$ is nuclear, we can choose a sequence of completely positive maps $\theta_1,\theta_2,\ldots$ from $B_n$ to $\A_{\infty}\cap \A '$ which admit a factorization via completely positive maps $\psi_j$ and $\varphi_j$ as in the following diagram:
$$
\xymatrix{
\B_n \ar[r]^{\psi_j} \ar@/_1pc/[rr]_{\theta_j} & M_k \ar[r]^{\varphi_j} &  \A_{\infty}\cap \A ' }  ,
$$
and such that
$\displaystyle \lim_{j \to \infty} \theta_j(x) = \gamma(x)$ for all $x \in B_n$. 
Using the Arveson extension theorem, for each $j=1,2,\ldots$ we can extend  $\psi_j$ to all of $B$. We write $\overline{\theta}_j$ for the composition of $\varphi_j$ with the chosen extension of $\psi_j$ to all of $\B$. Lift $\overline{\theta}_j$ to a completely positive map
$$
(\overline{\theta}_j(1),\overline{\theta}_j(2),\ldots) \colon \B \to l^{\infty}(\N,A)
$$
One checks that for a suitable increasing sequence $(n_j)_{j=1,2,\ldots}$, the composition of the map $(\overline{\theta}_1(n_1),\overline{\theta}_2(n_2),\ldots)  \colon \B \to l^{\infty}(A)$ with the quotient map onto $\A_{\infty}$ yields a completely positive map $\gamma' \colon \B \to \A_{\infty} \cap \A'$  which satisfies the first three conditions of the lemma.

Pick finite sets $F_n \subseteq B_n$ whose union is dense in $B$ and pick increasing finite subsets $G_n \subseteq G$ whose union is all of $G$. Choose maps $\varphi_n\colon B \to \A_{\infty} \cap \A'$ as in the statement, for $\eps = 1/n$, extended to $B$ as discussed above. For each such map we choose a completely positive contractive lifting to a map $\widetilde{\gamma}_n = (\gamma_n(1), \gamma_n(2),\ldots) \colon  B \to l^{\infty}(\N,\A)$. A standard diagonalization argument now yields an increasing sequence $(m_n)_{n \in \N}$ such that the map $(\gamma_1(m_1),\gamma_2(m_2),\ldots)$, composed with the quotient map, yields a map $B \to \A_{\infty} \cap \A'$ as required.

Assume now that $B$ is strongly self absorbing. 

The canonical inclusion $\A \hookrightarrow \A\rtimes_{\alpha}G$ induces an inclusion $\A_{\infty} \hookrightarrow (\A\rtimes_{\alpha}G)_{\infty}$. Pick a completely positive contraction $\gamma$ as in the statement. Composing with the canonical inclusion (and retaining the same notation), we can view $\gamma$ as a completely positive contraction from $B$ to $(\A\rtimes_{\alpha}G)_{\infty} \cap \A'$. 
Let $g \in G$, $a \in \A$ and $x \in \B$. Let $u_g \in M(\A\rtimes_{\alpha}G)$ be the canonical unitary corresponding to $g$. Then
$$
a u_g \gamma(x) = a \overline{\alpha}_g(\gamma(x)) u_g = a \gamma(x) u_g = \gamma(x) a u_g
 .
$$
Since the elements of the form $a u_g$ for $a \in \A$ and $g \in G$ span $\A \rtimes_{\alpha}G$, we find that $\gamma$ is a map into $(\A\rtimes_{\alpha}G)_{\infty} \cap (\A\rtimes_{\alpha}G) ' $, and that the conditions of \cite[Proposition 4.1]{HRW} are satisfied. Thus, $\A \rtimes_{\alpha}G$ is $B$-absorbing.
\end{proof}

\begin{Cor}
\label{Cor:enough-to-embed-several-order-zero-pairs}
Let $\A$ be a separable $C^*$-algebra. Let $G$ be a discrete countable group
with an action $\alpha\colon G \to \aut(\A)$. Let $d$ be a fixed natural number.
Suppose that for any $n$,  any $\eps>0$, and any finite set $G_0 \subseteq G$,
there are contractive order zero maps $\varphi_j^{(k)}\colon M_{n+j} \to 
\A_{\infty}\cap \A '$ for $j=0,1$ and $k=0,1,\ldots, d$, with commuting images, such that 
$\left \|\overline{\alpha}_g(\varphi_j^{(k)}(x)) - \varphi_j^{(k)}(x) \right \|<\eps\|x\|$ for all
$x \in M_{n+j}$ and all $g \in G_0$ and such that
$$
f= \sum_{k=0}^d \left [ \varphi_0^{(k)}(1) + \varphi_1^{(k)}(1) \right ]
$$ 
satisfies $f \leq 1$ and $af = a$ for all $a \in \A$. Then the full crossed
product $\A \rtimes_{\alpha} G$ is $\Zh$-absorbing.
\end{Cor}
\begin{proof}
Using the notation
of Corollary \ref{cor:d-order-zero-maps}, for $n \in \N$ we obtain an order zero map
$\Phi_n \colon \Zh_{n,n+1} \to \A_{\infty} \cap \A'$ satisfying the first two conditions
of Lemma \ref{lemma:enough-to-embed-building-block}. Since $\Phi_n$ is an order zero map,
we have $\Phi_n(1)\Phi_n(xy) =\Phi_n(x)\Phi_n(y)$ for all $x,y \in \Zh_{n,n+1}$. So, for all $a \in \A$, since
$a\Phi_n(1) = a$, we have
$$
a(\Phi_n(xy) - \Phi_n(x)\Phi_n(y)) = a(\Phi_n(1)\Phi_n(xy) -\Phi_n(x)\Phi_n(y)) = 0
$$
as required. Now write $\Zh$ as an inductive limit of $C^*$-algebras of the form $\Zh_{n_k,n_k+1}$ for an increasing sequence $(n_k)_{k \in \N}$. Apply 
Lemma \ref{lemma:enough-to-embed-building-block} to get the conclusion.  
\end{proof}

\begin{proof}[Proof of Theorem \ref{Thm:permanence-Z-finite-groups}]
 We show that the conditions of Corollary
\ref{Cor:enough-to-embed-several-order-zero-pairs} hold.

Let $r$ be a given positive integer. Fix two order zero maps $\theta_0\colon M_{r}
\to \Zh$ and $\theta_1\colon M_{r+1}
\to \Zh$ with commuting ranges such that 
$\theta_0(1)+\theta_1(1) = 1$. 

We claim that there are completely positive contractions $\iota_0,\iota_1,\ldots,\iota_d\colon \Zh \to \A_{\infty} \cap \A '$
satisfying:
\begin{enumerate}
\item \label{Z-fg-unital} 
$a\iota_k(1) = a$ for all $a \in \A$.
\item \label{Z-fg-multiplicative}
$a(\iota_k(xy) - \iota_k(x)\iota_k(y)) = 0$ for all $a \in \A$ and all $x,y \in
\Zh$.
\item $[\overline{\alpha}_g(\iota_k(x)),\iota_l(y)] = 0$ for all $k,l \in \{1,2,\ldots,d\}$ with $k < l$, for all $g \in G$, and for all $x,y \in \Zh$.
\end{enumerate}
First, use \cite[Proposition 4.1(d)]{HRW} to choose $\iota_0\colon \Zh \to \A_{\infty} \cap \A '$ satisfying conditions (\ref{Z-fg-unital}) and (\ref{Z-fg-multiplicative}) above. To get $\iota_1$, lift $\iota_0$ to a completely positive contraction $(\psi_1,\psi_2,\ldots) \colon \Zh \to l^{\infty}(\N,\A)$. Choose an increasing sequence $F_1 \subseteq F_2 \subseteq \ldots$ of finite subsets of $\Zh$ with dense union. Choose an increasing sequence $(n_j)_{j\in \N}$ such that
$$
\|[\alpha_g(\psi_j(x)),\psi_{n_j}(y)]\|<1/j
$$ 
for all $x,y \in F_j$ and all $g \in G$. Define $\iota_1$ to be composition of the map $(\psi_{n_1},\psi_{n_2},\ldots)$ with the quotient map to $\A_{\infty}$. One readily checks that $\iota_1$ satisfies the required conditions. Proceeding inductively, we construct $\iota_2,\iota_3,\ldots,\iota_d$ in a similar way.

Let $\left( f^{(l)}_g \right)_{l=0,1,\ldots,d\, ; g \in G}$ be a family of Rokhlin elements in $\A_{\infty}
\cap \A'$, as in Lemma
\ref{Lemma:central-sequence-reformulation-finite-group}, which is furthermore
chosen to commute with $\overline{\alpha}_g(\iota_k(\Zh))$ for all $g \in G$ and all
$k=0,1,\ldots,d$ (using Remark \ref{Rmk:central-sequence-reformulation-finite-group-Moreover}).) 

For $j=0,1$ and $x \in M_{r+j}$ define
$$
\theta_j^{(k)} (x)= \sum_{g \in G} f_{g}^{(k)}\overline{\alpha}_g(\iota_k \circ
\theta_j(x)) 
 .
$$
The images of these maps are clearly fixed by the action of $G$ on $\A_{\infty}
\cap \A'$, and have commuting images. Set 
$$
f = \sum_{k=0}^d \left [ \theta_0^{(k)}(1) + \theta_1^{(k)}(1) \right ]
 . 
$$
If $a \in \A$ then for $k=0,1,\ldots,d$ we have
$$
\iota_k(\theta_0(1) + \theta_1(1)) a = a
$$
and thus for all $g \in G$ we have
$$
\overline{\alpha}_g\big (\iota_k(\theta_0(1) + \theta_1(1))\big ) a = a
$$
as well. Therefore, for  $k=0,1,\ldots,d$,
$$
\left (\theta_0^{(k)}(1) + \theta_1^{(k)}(1) \right )a =  \sum_{g \in G}
f_{g}^{(k)}\overline{\alpha}_g\big (\iota_k ( \theta_0(1) + \theta_1(1))\big ) a
= \sum_{g \in G} f_{g}^{(k)} a
\, ,
$$
so
$$
fa =  \sum_{k=0}^d \sum_{g \in G} f_{g}^{(k)} a = a
.
$$
One similarly checks that $f \leq 1$. Thus the family of maps $(\theta_j^{(k)})_{j=0,1  ;  k=0,1,\ldots,d}$ satisfies the
conditions of Corollary \ref{Cor:enough-to-embed-several-order-zero-pairs}.
\end{proof}

We now consider the behavior of finite Rokhlin dimension under extensions.

\begin{Lemma}
\label{Lemma:G-fixed-approximate-id}
Let $\alpha \colon G \to \aut(A)$ be an action of a compact Hausdorff group $G$ on a $C^*$-algebra $\A$. Let $\J \lhd \A$ be an invariant ideal. Then there is a quasicentral approximate identity for $\J$ in $\A$ which is contained in the fixed point algebra $\J^G$.
\end{Lemma}
\begin{proof}
Choose a quasicentral approximate identity for $\J$ in $\A$ and average it over the group.
\end{proof}

\begin{Prop}
\label{Prop:ideal-and-quotient-finite-group}
Let $\alpha\colon G \to \aut(\A)$ be an action with $\dimrokct(\alpha)=d$. 
\begin{enumerate}
\item Suppose $\B \subseteq \A$ is a $G$-invariant hereditary subalgebra. Let $\beta$ be the restriction of $\alpha$ to $\B$. Then $\dimrokct(\beta) \leq \dimrokct(\alpha)$. 
\item Suppose
$\J\lhd \A$ is an $\alpha$-invariant ideal. Then the restriction of $\alpha$ to
$\J$ and the induced action on $\A/\J$ both have Rokhlin dimension with
commuting towers at most $d$.
\end{enumerate}

\end{Prop}
\begin{proof}
Let $F \subset \B$ be a finite subset, and let $\eps>0$. We may assume without loss of generality that $\|b\|\leq 1$ for all $b \in F$. The $C^*$-algebra $\B$ has an approximate identity in the fixed point subalgebra $\B^G$ (using Lemma \ref{Lemma:G-fixed-approximate-id} with $\B$ thought of as an ideal in itself). By picking an element sufficiently far out in such an approximate identity, we can choose a positive contraction $e \in \B^G$ such that $\|eb - b\| \leq \eps/5$ and $\|be - b\|<\eps/5$ for all $b \in F$. Choose a $(d,F\cup\{e\},\eps/5)$-Rokhlin system  $\left (f_g^{(k)}\right )_{k=0,1,\ldots,d \, ; g \in G}$ in $\A$. For $k=0,1,\ldots,d$ and for all $g \in G$ let 
$$
x_g^{(k)}=ef_g^{(k)}e \in \B
\, .
$$
It is straightforward to verify that the family $\left (x_g^{(k)}\right )_{k=0,1,\ldots,d \, ; g \in G}$ forms a $(d,F,\eps)$-Rokhlin system for the action $\beta$.

For the second part, the case of the quotient action on $\A/\J$ is immediate, and restricting to an ideal is a special case of restricting to a hereditary subalgebra.
\end{proof}

\begin{Thm}
\label{Thm:extensions-finite-group}
Let $\alpha\colon G \to \aut(\A)$ be an action of a finite group $G$ on a $C^*$-algebra $A$. Suppose $\J\lhd \A$ is an
$\alpha$-invariant ideal. Suppose the restriction of $\alpha$ to $\J$ and the
induced action on $\A/\J$  have Rokhlin dimensions with commuting towers
$d_{\J}$ and $d_{\A/\J}$. Then $\dimrokct(\alpha) \leq d_{\J} +
d_{\A/\J} +1$.
\end{Thm}
\begin{proof}
We denote by $\pi\colon  \A \to \A/\J$ the quotient map, and by $\overline{\alpha}$ the
quotient action on $\A/\J$. Let $F \subseteq \A$ be a finite subset and let $\eps>0$. We assume without loss of generality that $\|a\| \leq 1$ for all $a \in F$.
Pick a $(d_{\A/\J},\pi(F),\eps/7)$-Rokhlin system $\left (b_g^{(k)}\right )_{k=0,1,\ldots,d_{\A/\J}\, ; g \in G}$ in $\A/\J$. Using Remark \ref{Rmk:finite-group-ortho}, we
may assume without loss of generality that  $b_h^{(k)}b_g^{(k)} =0$ for $k=0,1,\ldots,d$ and all $g,h\in G$ with $g \neq h$. 
Since the cone over $\C^n$ is projective, there are $x_g^{(k)}\in \A$, for $g \in G$ and $k=0,1,\ldots,d$, such that $0\leq x_g^{(k)} \leq 1$  and $\pi (x_g^{(k)}) = b_g^{(k)}$, and moreover $x_g^{(k)}x_h^{(k)} = 0$ whenever $g \neq h$.

For any $a
\in F$, for any $g\in G$, and for $k=0,1,\ldots,d$ we have
$\textrm{dist}\left ([x_g^{(k)},a],\J \right )<\eps/7$. So if $(e_{\lambda})_{\lambda \in \Lambda}$ is an approximate
identity for $\J$ then $\lim_{\lambda}\left \|[x_g^{(k)},a](1-e_{\lambda})\right \| < \eps/7$. Also, we have 
$$
\lim_{\lambda}\left \|\left (1-\sum_{k=0}^{d_{\A/\J}}\sum_{g \in G}
x_g^{(k)}\right ) a(1-e_{\lambda})\right \| =0
\quad {\mbox{and}} \quad 
\lim_{\lambda}\left \|[x_g^{(k)},x_h^{(l)}]a(1-e_{\lambda}) \right \|  =0 \,.
$$

By taking an element far enough out in a $G$-invariant quasicentral approximate identity for $J$ in $\A$ (see Lemma \ref{Lemma:G-fixed-approximate-id}), we can find 
 a positive contraction $q \in \J$ which satisfies the
following conditions:
\begin{enumerate}
\item
\label{2.10-quasicentral-q}
 $\displaystyle \|qa - aq\|<\frac{\eps}{21\cdot \# G \cdot (d_{\A/\J}+1)}$ for all
$$
a \in F \cup \left
\{ \alpha_h(x_g^{(k)}),\alpha_h\left (\sqrt{x_g^{(k)}} \right ) \mid g,h \in G, \;
k=0,1,\ldots,d_{\A/\J} \right \}
$$
\item 
\label{2.10-q-condition-2}
$\left \|[x_g^{(k)},a](1-q) \right \|<\eps/7$   for all $a \in F$.
\item $\alpha_g(q) = q$ for all $g \in G$.
\item $\displaystyle \left \| \left (1-\sum_{k=0}^{d_{\A/\J}}\sum_{g \in G}
x_g^{(k)}\right )a(1-q)\right \|<\eps/7$ for all $a \in F$.
\item 
\label{2.10-q-condition-5}
$\left \|\left (\alpha_h(x_g^{(k)}) - x_{hg}^{(k)} \right )a(1-q)\right \| <\eps/7$ for all $h,g \in G$, for all $a \in F$, and for 
$k=0,1,\ldots,d_{\A/\J}$.
\item $\left \|\left [x_g^{(k)},x_h^{(l)}\right ]a(1-q) \right \| < \eps/7$ for all $g,h \in G$, for all $a \in F$, and for $k,l
=0,1,\ldots,d_{\A/\J}$.
\end{enumerate}

Now, fix a finite set $\widetilde{F} \subseteq \J$ such that for all $a \in F$, all
$g\in G$, and $k=0,1,\ldots,d_{\A/\J}$ we have  $\textrm{dist}([x_g^{(k)},a],\widetilde{F})<\eps/7$. Let
$$
F_{\J} = \widetilde{F} \cup \{qa \mid a \in F\} \cup \{q\} \cup \{qx_g^{(k)} \mid g
\in G, k=0,1,\ldots,d_{\A/\J}\} 
\, .
$$
Set $M = \max \{\|f\| \mid f \in F_{\J}\} $. Choose $\delta \in (0,\eps/42)$ such that 
whenever $\B$ is a $C^*$-algebra and $b,c \in \B$ satisfy $0 \leq b \leq 1$, $\|c\|\leq M$, and $\|bc-cb\|<\delta$, then 
$$
\|b^{1/2}c - cb^{1/2}\| < \frac{\eps}{21\cdot \# G \cdot (d_{\J}+1)} \, .
$$
(That $\delta$ exists follows by approximating $b^{1/2}$ with polynomials in $b$.)

Choose a $(d_{\J},F_{\J},\delta)$-Rokhlin
system $\left (y_g^{(k)} \right )_{k=0,1,\ldots,d_J \, ; g \in G}$ in $\J$. We assume again that $y_g^{(k)}y_h^{(k)} = 0$ for $g
\neq h$. For $g \in G$ and $k=0,1,\ldots,d_{\J}$, we then have 
\begin{equation}
\label{2.10-y,q}
\left \| \left [\sqrt{y_g^{(k)}},f \right ] \right \|<\frac{\eps}{21\cdot \# G \cdot (d_{\J}+1)}
\end{equation} 
for all $f \in F_{\J}$.
Set 
\begin{equation}
\label{2.10-(*)}
z_g^{(k)} = \sqrt{y_g^{(k)}}q\sqrt{y_g^{(k)}} \, .
\end{equation}
Notice that 
\begin{equation}
\label{2.10-(**)}
\left \|z_g^{(k)} - y_g^{(k)}q \right \| < \frac{\eps}{21\cdot \# G \cdot (d_{\J}+1)}
\, .
\end{equation}

It follows now that  for all $g,h
\in G$, and $k=0,1,\ldots,d$ we have, using (\ref{2.10-(**)}) in the first step and $\delta<\eps/42$ in the last step,
\begin{align}
\label{2.10-E1}
\|\alpha_g(z_h^{(k)}) - z_{gh}^{(k)}\| & <  
\|\alpha_g(y_h^{(k)}q) - y_{gh}^{(k)}q\| + \frac{2\eps}{21(d_{\J} + 1)} 
\\ \nonumber
& =  
\|(\alpha_g(y_h^{(k)}) - y_{gh}^{(k)})q\| + \frac{2\eps}{21(d_{\J} + 1)}
<\delta + \frac{2\eps}{21(d_{\J} + 1)} < \frac{\eps}{7}
\, .
\end{align}

We write $x \approx_{\eps} y$ to mean $\|x-y\|<\eps$. 
For any $a \in F$, we have, using (\ref{2.10-(**)}) at the first step,
\begin{align}
\label{2.10-E7}
z_g^{(k)}a - az_g^{(k)} & 
  \approx_{\eps/7}   \;\;  y_g^{(k)}qa - qay_g^{(k)} \\ \nonumber
{} & \approx_{\eps/7}   \;\;  y_g^{(k)}qa - y_g^{(k)}qa = 0 \, ,
\end{align}
and, by using (\ref{2.10-(**)}),
\begin{equation}
\label{2.10-4X23_1}
\sum_{g \in G}\sum_{k=0}^{d_J} z_g^{(k)} \approx_{\eps/7} \sum_{g \in G}\sum_{k=0}^{d_J} y_g^{(k)}q \approx_{\eps/7} q  \;\; .
\end{equation}
Furthermore, for any $g,h \in G$ and $k,l=0,1,\ldots,d_{\J}$ we have
\begin{align}
\label{2.10-NewE}
\left \|z_g^{(k)}z_h^{(l)} - z_h^{(l)}z_g^{(k)} \right \|  & \leq
\left \|z_g^{(k)} - qy_g^{(k)} \right \| +
\left \|z_h^{(l)} - y_h^{(l)}q \right \| 
\\ \nonumber {} &  + 
\left \|z_h^{(l)} - qy_h^{(l)} \right \| +
\left \|z_g^{(k)} - y_g^{(k)}q \right \| +
\left \|q y_g^{(k)} y_h^{(l)} q - q y_h^{(l)} y_g^{(k)} q \right \| \, .
\end{align}
By (\ref{2.10-(**)}) and selfadjointness, each of the first four terms is less than $\eps/21$. The last term is at most $\left \| q \left [y_g^{(k)},y_h^{(l)} \right ] q \right \|<\delta$. Therefore
\begin{equation}
\label{2.10-E2}
\left \| \left [z_g^{(k)},z_h^{(l)} \right ] \right \|
 \leq \frac{4\eps}{21} + \delta <
\frac{5\eps}{21}  
\, .
\end{equation}
Set  $\widetilde{x}_g^{(k)} = \sqrt{x_g^{(k)}}(1-q)\sqrt{x_g^{(k)}}$. For $g,h \in G$ with
$g \neq h$ and $k=0,1,\ldots,d_{\A/\J}$, we have
\begin{equation}
\label{2.10-E8}
\widetilde{x}_g^{(k)}\widetilde{x}_h^{(k)} = 0 \, .
\end{equation} 
Now, by condition (\ref{2.10-quasicentral-q}) from the list of conditions the element $q$ was chosen to satisfy,
\begin{equation}
\label{2.10-xEst}
\|\widetilde{x}_g^{(k)} -  x_g^{(k)}(1-q) \| < \frac{\eps}{21\cdot \# G \cdot (d_{\A/\J}+1)}
\, ,
\end{equation}
so for all $a \in F$, for all $g,h \in G$, and for $k=0,1,\ldots,d_{\A/\J}$
we have
\begin{align}
\label{2.10-E5}
\|(\alpha_h(\widetilde{x}_g^{(k)}) - \widetilde{x}_{hg}^{(k)})a\| & < \|(\alpha_h(x_g^{(k)})
- x_{hg}^{(k)})(1-q)a\|+\frac{2\eps}{21} 
\\ \nonumber
{}& < 
\|(\alpha_h(x_g^{(k)})
- x_{hg}^{(k)})a(1-q)\|+\frac{2\eps}{21} + \frac{\eps}{21(d_{\A/\J} +1)}  
<
\frac{2\eps}{7}  
.
\end{align}
Similarly, at the second step using (\ref{2.10-xEst}) and conditions (\ref{2.10-quasicentral-q}) and (\ref{2.10-q-condition-2}) from the list of conditions $q$ was chosen to satisfy,
\begin{align}
\label{2.10-x,a}
\left \| \left [\widetilde{x}_g^{(k)},a \right ] \right \| 
\leq & 2 \left \| \widetilde{x}_g^{(k)} - x_g^{(k)}(1-q) \right \|
+
\|(1-q) a - a (1-q)\| + 
\left \| (x_g^{(k)}a - ax_g^{(k)})(1-q) \right \| \\ \nonumber
< & \frac{2 \eps}{21 \cdot \# G \cdot (d_{A/J}+1)} + \frac{\eps}{7} + \frac{\eps}{7} <
\eps \, ,
\end{align}
and
\begin{align}
\label{2.10-4X23_2}
  \left \| \left (\sum_{k=0}^{d_{\A/\J}}\sum_{g \in G} \widetilde{x}_g^{(k)}\right )a
- a(1-q)\right \| 
 & {} \\ \nonumber
  {}  & \kern-8em \leq 
\sum_{k=0}^{d_{\A/\J}}\sum_{g \in G} \left \| \widetilde{x}_g^{(k)} - x_g^{(k)}(1-q)
\right \| 
\\ \nonumber {} & \kern-6em
 +  \left \| \sum_{k=0}^{d_{\A/\J}}\sum_{g \in G} x_g^{(k)} \right \| \cdot
\|(1-q)a - a(1-q)\|
+
\left \| \sum_{k=0}^{d_{\A/\J}}\sum_{g \in G} x_g^{(k)} a(1-q) \right \| 
\\ \nonumber {}& \kern-8em
\leq   
(1+d_{\A/\J})\cdot \# G \cdot \frac{\eps}{21 \cdot \# G \cdot (d_{\A/\J} + 1)} 
\\ \nonumber {}& \kern-6em + 
 (1+d_{\A/\J})\cdot \# G \cdot \frac{\eps}{21\cdot \# G \cdot (d_{\A/\J} + 1)} + \frac{\eps}{7} \\ \nonumber
 {}& \kern-8em <\frac{2\eps}{7}
\end{align}
for all $a \in F$. For the approximate commutation condition, for all $a \in F$ we have, using an argument similar to (\ref{2.10-NewE}) at the first step, $\|1-q\|\leq 1$ at the second step, and $\left \| \left [ x_g^{(k)}, x_h^{(l)} \right ] \right \|\leq 2$ and condition (\ref{2.10-q-condition-5}) from the list of conditions for the choice of $q$ at the third step,
\begin{align}
\label{2.10-E3}
\left \| \left [\widetilde{x}_g^{(k)},\widetilde{x}_h^{(l)} \right ]a \right \| 
& \leq 
\left \|(1-q)\left [x_g^{(k)},x_h^{(l)} \right ] (1-q)a \right \| + \frac{4\eps}{21} 
\\ \nonumber
{} & \leq 
\left \| \left [x_g^{(k)},x_h^{(l)} \right ] \big ( (1-q)a - a(1-q) \big ) \right \| + \left \|\left [x_g^{(k)},x_h^{(l)} \right ] a (1-q) \right \|+  \frac{4\eps}{21} 
\\ \nonumber
{} & < 2 \cdot \frac{\eps}{21 \cdot \# G \cdot (d_{A/J}+1)} + \frac{\eps}{7} +  \frac{4\eps}{21} 
<  \frac{3\eps}{7}  
\, .
\end{align}
Now, for all $a \in F$, for $g,h \in G$, for $k=0,1,\ldots,d_{\A/\J}$, and for $l=0,1,\ldots,d_{\J}$ we have, by commuting the term $(1-q)$ all the way to the right in the second step, and by commuting the term $q$ to the left three terms in the first summand and one term to the right in the second summand, and using the bounds from condition (\ref{2.10-quasicentral-q}) from the list of conditions $q$ was chosen to satisfy and from inequality (\ref{2.10-y,q}):
\begin{align*}
\label{2.10-E4}
\left ( \widetilde{x}_g^{(k)} z_h^{(l)} -  z_h^{(l)} \widetilde{x}_g^{(k)}
 \right )  a 
  {}&
   \\ {}& \kern-7em
 =
 \left (
\sqrt{x_g^{(k)}} (1-q) \sqrt{x_g^{(k)}}\cdot \sqrt{y_h^{(l)}} q \sqrt{y_h^{(l)}} -
 \sqrt{y_h^{(l)}} q \sqrt{y_h^{(l)}}\cdot \sqrt{x_g^{(k)}} (1-q) \sqrt{x_g^{(k)}}
 \right ) a 
 \\ {}& \kern-7em
\approx_{2\eps/21+2\delta} 
 \left (
\sqrt{x_g^{(k)}} \cdot \sqrt{x_g^{(k)}}\cdot \sqrt{y_h^{(l)}} q \sqrt{y_h^{(l)}} -
 \sqrt{y_h^{(l)}} q \sqrt{y_h^{(l)}}\cdot \sqrt{x_g^{(k)}} \cdot \sqrt{x_g^{(k)}}
 \right ) (1-q)a
\\ {}&  \kern-7em
\approx_{2\eps/21+2\delta} 
 \left (
(qx_g^{(k)}) y_h^{(l)} -
 y_h^{(l)} (q x_g^{(k)})
 \right ) (1-q)a
\end{align*}
and 
$$
 \left \|\left [qx_g^{(k)},y_h^{(l)} \right ]\right \| \cdot \|(1-q)a\|  < \delta
\, .
$$
Therefore, 
\begin{equation}
\label{2.10-E4}
\left \| \left [ \widetilde{x}_g^{(k)}, z_h^{(l)} \right ] a \right \| < \frac{4\eps}{21} + 5\delta < \eps \, .
\end{equation}
Finally, for any $a \in F$, 
\begin{align}
\label{2.10-E6}
 {}& \left \| \left (\sum_{k=0}^{d_{\A/\J}}\sum_{g \in G} \widetilde{x}_g^{(k)}  +
\sum_{k=0}^{d_{\J}}\sum_{g \in G} z_g^{(k)} \right) a - a\right \|   \\ \nonumber
{}& \quad \leq  
 \left \| \left (\sum_{k=0}^{d_{\A/\J}}\sum_{g \in G} \widetilde{x}_g^{(k)}\right )a
- a(1-q)\right \| +
\left \| \left ( \sum_{k=0}^{d_{\J}}\sum_{g \in G} z_g^{(k)} \right )  - q
\right \| 
+
\|qa-aq\| \\ \nonumber
{}& \quad
< \frac{2\eps}{7} + \frac{2\eps}{7} + \frac{\eps}{21\cdot\# G \cdot (d_{A/J}+1)} < \eps
\, .
\end{align}
For $g \in G$ and $k=0,1,\ldots,d_{\A/\J} + d_{\J} + 1$, set
$$
f_g^{(k)} = \left \{ 
\begin{matrix}
\widetilde{x}_g^{(k)} & \mid & k=0,1,\ldots,d_{\A/J} \\
z_g^{(k-d_{\A/\J} -1 )} & \mid & k=d_{\A/\J} + 1,d_{\A/\J} + 2,\ldots,d_{\A/\J} + d_{\J} + 1   
\, .
\end{matrix}
\right .
$$

One checks now that $\left (f_g^{(k)}\right)_{k=0,1,\ldots,d_{\A/\J}+d_{\J}+1}$ is a
$(d_{\A/\J}+d_{\J}+1,F,\eps)$-Rokhlin system, as follows.
\begin{itemize}
\item Condition (\ref{def-finite-group-ortho}) in Definition \ref{def: nonunital Rokhlin finite groups} follows from (\ref{2.10-E8}) and the fact that $z_g^{(k)}z_h^{(k)} = 0$ for $k=0,1,\ldots,d_J$ and all $g,h \in G$ with $g \neq h$.
\item Condition (\ref{def-finite-group-unit}) follows from (\ref{2.10-E6}).
\item Condition (\ref{def-finite-group-central}) follows from (\ref{2.10-x,a}) and  (\ref{2.10-E7}).
\item Condition (\ref{def-finite-group-permuted}) follows from (\ref{2.10-E5}) and (\ref{2.10-E1}).
\item Condition (\ref{commuting tower assumption}) follows from (\ref{2.10-E3}), (\ref{2.10-E2}), and (\ref{2.10-E4}).
\end{itemize}
This completes the proof that  $\dimrokct(\alpha) \leq
d_{\A/\J} + d_{\J} + 1$, as required.
\end{proof}

We conclude this section by applying the results above concerning equivariant extensions to the case of actions on type I $C^*$-algebras. We first need the following proposition.
\begin{Prop}
\label{prop:Hausdorff-spectrum}
Let $\A$ be a separable $C^*$-algebra with Hausdorff primitive ideal space $X$. Suppose $G$ is a finite group, and $\alpha \colon G \to \A$ is an action. Suppose $\alpha$ induces a free action on $X$, and suppose $X$ has covering dimension $d$. Then $\dimrokct(\alpha) \leq d$. 
\end{Prop}
\begin{proof}
By the Dauns-Hoffman theorem, we can identify $C_b(X)$ with the center of $M(\A)$. By slight abuse of notation, we denote by $\alpha$ the extension of the action of $\alpha$ to $C_b(X)$ and to $C_0(X)$. Let $X_1 \subseteq X_2 \subseteq \cdots \subseteq X$ be an increasing sequence of $G$-invariant compact subsets such that $\bigcup_n X_n = X$. Each $X_n$ has covering dimension at most $d$. We can view $G$ as acting on $C(X_n)$ as well, and by Lemma \ref{Lemma:X-Rokhlin-implies-finite-Rokhlin-dim}, those actions have Rokhlin dimension at most $d$ with commuting towers.

Let $F \subseteq \A$ be a finite subset, and let $\eps>0$. By perturbing $F$, we may assume without loss of generality that there is an $N$ such that $C_0(X \smallsetminus X_N)a = 0$ for all $a \in F$. 
Let $(f_g^{(l)})_{l=0,1,\ldots,d \, ; g \in G}$ be a family of Rokhlin elements in $C(X_N)$ as in Definition \ref{def: positive Rokhlin finite groups}. Extend each  of them to positive a contraction in $C_0(X)$. We retain the same notation for the extensions of those elements. Let $e \in A$ be a $G$-invariant positive contraction in $\A$ which satisfies 
$\|ea - a\| < \eps/2$ and $\|ae-a\|<\eps/2$ for all $a \in F$. One readily verifies that the family $(f_g^{(l)}e)_{l=0,1,\ldots,d \, ; g \in G}$ satisfies the conditions of Definition \ref{def: nonunital Rokhlin finite groups}.
\end{proof}

\begin{Cor}
\label{Cor:type-I}
Let $\A$ be a separable type I $C^*$-algebra with finite composition series $\A = I_n \rhd I_{n-1} \rhd \cdots \rhd I_1 \rhd I_0 = 0$, in which $I_{k+1}/I_{k}$ has Hausdorff primitive ideal space. Suppose the primitive ideal space of $I_{k}/I_{k-1}$ has finite covering dimension for $k=1,2,\ldots,n$. Let $G$ be a finite group, and let $\alpha \colon G \to \aut(\A)$ be an action. If $G$ acts freely on the primitive ideal space of $A$ then $\dimrokct(\alpha) <\infty$.
\end{Cor}
\begin{proof}
The proof of \cite[Corollary 8.1.2]{Ph1} shows that the composition series above can be chosen to be $G$-invariant. By Proposition \ref{prop:Hausdorff-spectrum}, the induced actions of $G$ on the algebras $I_{k+1}/I_k$ have finite Rokhlin dimension with commuting towers. The statement now follows by repeated application of Theorem \ref{Thm:extensions-finite-group}.
\end{proof}

Corollary \ref{Cor:type-I} is a partial converse to the following.
\begin{Prop}
Let $\A$ be a separable type I $C^*$-algebra. Let $G$ be a finite group, and let $\alpha\colon G \to \aut(\A)$ be an action. If $\dimrokct(\alpha)<\infty$ then the induced action of $\alpha$ on the primitive ideal space of $\A$ is free. 
\end{Prop}
\begin{proof}
Suppose not. Then there is a finite cyclic subgroup $H$ of $G$ of prime order such that the restriction of the action to $H$ is not free either. By Lemma \ref{Lemma:finite-subgroup}, we have $\dimrokct(\alpha|_H)<\infty$ as well. By \cite[Lemma 8.2.1]{Ph1}, there exist $H$-invariant ideals $I \rhd J$ such that $I/J$ is isomorphic to the $C^*$-algebra of compact operators on some Hilbert space. Thus, $H$ acts on $I/J$ via inner automorphisms, but also has finite Rokhlin dimension. By Lemma \ref{Lemma:Rokhlin-pointwise-outer}, this cannot happen.
\end{proof}

\section{Actions of a single automorphism: permanence properties}
\label{sec:single-auto-permanence-properties}

In this section we discuss the analogs of the results from Section
\ref{sec:finite-group-permanence-properties} for actions of $\Z$. We follow the
same organization as the previous section, and begin by extending the permanence
properties from \cite{HWZ} to the nonunital setting. We state these as
Theorems \ref{Thm:permanence-dimension-single-auto} and
\ref{Thm:permanence-Z-single-auto}.

\begin{Thm}
\label{Thm:permanence-dimension-single-auto}
Let $\A$ be a  $C^*$-algebra with finite nuclear dimension and $\alpha \in
\aut(\A)$
be an automorphism with $\dimrokct(\alpha)=d$. Then the crossed product $\A
\rtimes_{\alpha} \Z$ has finite 
nuclear dimension as well --- in fact, 
$$
\dimnuc (\A \rtimes_{\alpha} \Z) \leq 4(\dimnuc(\A) +1) (d+1)-1  \; .
$$ 
\end{Thm}

\begin{Thm}
\label{Thm:permanence-Z-single-auto}
Let $\A$ be a separable $\Zh$-absorbing $C^*$-algebra and $\alpha \in \aut(\A)$
be an automorphism with $\dimrokct(\alpha)<\infty$. Then  $\A \rtimes_{\alpha}
\Z$ is $\Zh$-absorbing as well.
\end{Thm}

Theorem \ref{Thm:permanence-dimension-single-auto} is a generalization of \cite[Theorem 4.1]{HWZ} to the nonunital setting. The modification required to obtain this generalization is straightforward and will be omitted.
For Theorem
\ref{Thm:permanence-dimension-single-auto}, it is not necessary to assume that
the different Rokhlin towers approximately commute. 

For the proof of Theorem \ref{Thm:permanence-Z-single-auto} we need
modifications similar to those we used in the case of finite group
actions. We skip details in those parts of the proof which closely mirror those which appear in the proof of Theorem \ref{Thm:permanence-Z-finite-groups}.

\begin{proof}[Proof of Theorem \ref{Thm:permanence-Z-single-auto}]
 We show that the conditions of Corollary
\ref{Cor:enough-to-embed-several-order-zero-pairs} hold.

Let $r$ be a given positive integer. Fix two order zero maps $\theta_j\colon M_{r+j}
\to \Zh$, for $j=0,1$, with commuting ranges, such that 
$\theta_0(1)+\theta_1(1) = 1$. Let $K$ be the union of the images of the unit
balls of $M_r$ and $M_{r+1}$ under these maps.
We claim that there are completely positive contractions $\iota_0,\iota_1,\ldots,\iota_d,\mu_0,\mu_1,\ldots,\mu_d\colon \Zh \to
\A_{\infty} \cap \A '$ satisfying: 
\begin{enumerate}
\item $a\iota_k(1) = a$ and  $a\mu_k(1) = a$ for all $a \in \A$.
\item $a(\iota_k(xy) - \iota_k(x)\iota_k(y)) = 0$ and $a(\mu_k(xy) -
\mu_k(x)\mu_k(y)) = 0$ for all $a \in \A$ and all $x,y \in \Zh$.
\item The image of each of those maps commutes with all the iterates of the image of any of the other maps under $\overline{\alpha}$. 
\end{enumerate}
The proof of this claim is very similar to the analogous one in the proof of Theorem \ref{Thm:permanence-Z-finite-groups}, and we omit it. Notice that if $\iota \colon \Zh \to \A_{\infty} \cap \A'$ is a homomorphism, then $\bigcup_{n \in \Z} \overline{\alpha}^n(\Zh)$ is a separable set, and therefore replacing the finite group $G$ from Theorem \ref{Thm:permanence-Z-finite-groups} by the group of integers causes no difficulties.

Since the tensor flip in $\Zh \otimes \Zh$ is approximately inner, there is $w$ in the unitary group 
$U(\Zh \otimes \Zh)$ such that
$$\|w(x\otimes 1)w^*-1 \otimes x\| < \frac{\eps}{4}$$ for all $x \in K$. The unitary group of the Jiang-Su algebra is connected. 
Thus, $w$ can be connected to $1$ via a rectifiable path.
Let $L$ be the length of such a path. Choose $n\in \N$ such that
$L\|x\|/n <\eps/8$ for all $x \in F$.

Recall (\cite[Exercise 3.5.1]{Brown-Ozawa}) that if $B_1$, $B_2$, and $D$ are $C^*$-algebras, and
$T_1\colon B_1 \to D$ and $T_2\colon B_2 \to D$ are completely positive contractions with commuting ranges, then there exists a completely positive contraction $T \colon B_1 \otimes_{\max} B_2 \to D$ such that $T(b_1 \otimes b_2) = T_1(b_1)T_2(b_2)$ for all $b_1 \in B_1$ and $b_2 \in B_2$. Therefore, the commutation properties of the maps $\iota_k$ and $\mu_k$ for $k=0,1,\ldots,d$ imply that there are
 completely positive contractions 
$$
\rho_k,\rho_k'\colon \Zh \otimes \Zh \to
\A_{\infty} \cap \A' 
$$ 
such that for all $x,y\in \Zh$,
$$
\rho_k(x \otimes y) = \iota_k (x) \mu_k(y)\quad \mathrm{and} \quad \rho_k'(x\otimes y) =
\overline{\alpha}^n(\iota_k(x))\mu_k(y) 
\, .
$$ 
 We claim that
$\rho(1)a = a$ and $(\rho(xy) - \rho(x)\rho(y))a =0$ for all $a \in \A$ and $x,y
\in \Zh \otimes \Zh$, and the same  for $\rho'$. This evidently
holds for elementary tensors in $\Zh \otimes \Zh$, and by linearity and
continuity holds for all $x,y \in \Zh \otimes \Zh$, proving the claim.

Pick unitaries $1=w_0,w_1,\ldots,w_n = w$ in the identity component 
$U_0(\Zh \otimes \Zh)$ of $U(\Zh \otimes \Zh)$ such that $\|w_j-w_{j+1}\| \leq L/n$ for
$j=0,1,\ldots,n-1$. Now, for $k=0,1,\ldots,d$, let 
$$
x_j^{(k)} = \rho_k(w_j)^*\rho_k'(w_j)
\, .
$$ 
The elements $x_j^{(k)}$ behave like unitaries when multiplied by elements from $\A$, that is, we
have
$$
x_j^{(k)}\left (x_j^{(k)} \right )^*a =\left (x_j^{(k)} \right )^*x_j^{(k)}a = a
$$
for all $a \in \A$. Furthermore, $x_0^{(k)}a = a$ for all $a \in \A$ and $k=0,1,\ldots,d$. 
Note also that
$$
\left \|x_j^{(k)} - x_{j+1}^{(k)} \right \| \leq \frac{2L}{n}
$$ 
for $j=0,1,\ldots,n-1$, and  
that 
\begin{align*}
{}& 
\left \|\left ( x_n^{(k)} \overline{\alpha}^n(\iota_k(y)) \left (x_n^{(k)} \right )^* -
\iota_k(y) \right ) a \right \| 
 \\ {}& \quad =
\left \|\left (\rho_k(w)^*\rho_k'(w) \overline{\alpha}^n(\iota_k(y)) \rho_k'(w)^*\rho_k(w) -
\iota_k(y) \right ) a \right \| 
\\ {}& \quad < 
\left \|\left (\rho_k(w)^* \mu_k(y)\rho_k(w) -
\iota_k(y) \right ) a \right \| + \frac{\eps}{4}
<
\frac{\eps}{2}
\end{align*}
for all $y\in K$ and all $a \in \A$ with norm at most 1.

Likewise, pick unitaries $1=w'_0,w'_1,\ldots,w'_{n+1} = w' \in
U_0(\Zh \otimes \Zh)$ such that $\|w_j-w_{j+1}\| \leq L/(n+1)$ for
$j=0,1,\ldots,n$, and for $k=0,1,\ldots,d$, let 
$$
y_j^{(k)} = \rho_k(w_j')^*\rho_k'(w_j')
\, .
$$
The elements $y_j^{(k)}$ satisfy the analogous properties to those of the elements $x_j^{(k)}$, with $n+1$ in place of $n$.

Let $f_{0,0}^{(l)},\ldots,f_{0,n-1}^{(l)}, f_{1,0}^{(l)},\ldots,f_{1,n}^{(l)}
 \in \A_{\infty} \cap \A '$ , for $l = 0,1,\ldots,d$, be
commuting Rokhlin elements in $\A_{\infty} \cap \A '$  as  in Lemma
\ref{Lemma:central-sequence-reformulation-single-auto} which are furthermore
chosen to commute with $\overline{\alpha}^j(\iota_k(\Zh))$ and
$\overline{\alpha}^j(\mu_k(\Zh))$ for $k=0,1,\ldots,d$ and for all $j\in \Z$ (see Remark \ref{Rmk:central-sequence-reformulation-single-auto-Moreover}).

Now, for $i=0,1$, set
\begin{align*}
{}& \theta_i^{(k)} (x)= 
\sum_{j=0}^{n-1}f_{0,j}^{(k)}\overline{\alpha}^{j-n}(x_j^{(k)})\overline{\alpha}^j
(\iota_k \circ \theta_i(x))\overline{\alpha}^{j-n}\left (\left (x_j^{(k)} \right )^* \right ) 
\\ {}&  \quad \quad \quad \quad
+ 
\sum_{j=0}^{n}f_{1,j}^{(k)}\overline{\alpha}^{j-n}(y_j^{(k)})\overline{\alpha}^j
(\iota_k \circ \theta_i(x))\overline{\alpha}^{j-n}\left (\left (y_j^{(k)} \right )^* \right )
\, .
\end{align*}

We can check that 
$$
\left \|\left ( \overline{\alpha}(\theta_i^{(k)}(x)) - \theta_i^{(k)}(x) \right )a \right \|<\eps
$$
for all $a$ in the unit ball of $\A$ and for all $x$ in the unit balls of $M_r$, $M_{r+1}$, respectively.
Furthermore, for all $a \in \A$,
$$
\left (\theta_0^{(k)}(1) + \theta_1^{(k)}(1) \right )a = \left (
\sum_{j=0}^{n-1}f_{0,j}^{(k)} + \sum_{j=0}^{n}f_{1,j}^{(k)} \right ) a
$$ 
and thus, if we denote 
$$
f = \sum_{k=0}^d\sum_{i=0}^1\theta_i^{(k)}
$$ 
we get that $fa = a$ for all $a \in \A$. 
Therefore those maps satisfy the conditions of Corollary
\ref{Cor:enough-to-embed-several-order-zero-pairs}.
\end{proof}

We now consider the analogs of Proposition
\ref{Prop:ideal-and-quotient-finite-group} and Theorem
\ref{Thm:extensions-finite-group} concerning equivariant extensions for the case
of a single automorphism. The idea of the proof is similar. The main difference
is that we cannot expect to have quasicentral approximate identities that are
fixed under the automorphism. However we can have ones that are approximately
fixed. Although the proofs are otherwise quite similar, we provide most of the
details for the reader's convenience.

\begin{Lemma}
\label{Lemma:almost-fixed-approximate-id}
Let $\alpha \colon G \to \aut(A)$ be an action of a discrete amenable group $G$ on a $C^*$-algebra $\A$. Let $\J \lhd \A$ be an invariant ideal. For any finite set of elements $G_0 \subseteq G$ and any $\eps>0$ there exists a quasicentral approximate identity $(e_{\lambda})_{\lambda \in \Lambda}$ for $\J$ in $\A$ such that $\|\alpha_g(e_{\lambda}) - e_{\lambda}\|<\eps$ for all $g \in G_0$ and all $\lambda \in \Lambda$.
\end{Lemma}
\begin{proof}
Choose a quasicentral approximate identity for $\J$ in $\A$ and average it over a sufficiently large F{\o}lner set.
\end{proof}

\begin{Prop}
\label{Prop:ideal-and-quotient-single-auto}
Let $\alpha\in\aut(\A)$ be an automorphism with $\dimrokct(\alpha)=d$. 
\begin{enumerate}
\item Suppose $\B \subseteq \A$ is a $\alpha$-invariant hereditary subalgebra. Let $\beta$ be the restriction of $\alpha$ to $\B$. Then $\dimrokct(\beta) \leq \dimrokct(\alpha)$. 
\item Suppose
$\J\lhd \A$ is an $\alpha$-invariant ideal. Then the restriction of $\alpha$ to
$\J$ and the induced automorphism on $\A/\J$ both have Rokhlin dimension with
commuting towers at most $d$.
\end{enumerate}
\end{Prop}
\begin{proof}[Outline of proof]
The proof is very similar to that of Proposition \ref{Prop:ideal-and-quotient-finite-group}. The main change required is that the element $h \in \B$ cannot be chosen to be fixed under $\alpha$, and instead it needs to be chosen so that $\|\alpha(h) - h\|$ is sufficiently small (as in Lemma \ref{Lemma:almost-fixed-approximate-id}). We omit the details.
\end{proof}

\begin{Thm}
Let $A$ be a $C^*$-algebra and let $\alpha\in \aut(\A)$. Suppose $\J\lhd \A$ is an
$\alpha$-invariant ideal. Suppose the restriction of $\alpha$ to $\J$ and the
induced action on $\A/\J$ both have Rokhlin dimensions with commuting towers
$d_{\J},d_{\A/\J}$, respectively. Then $\dimrokct(\alpha) \leq d_{\J} +
d_{\A/\J} +1$.
\end{Thm}
\begin{proof}
To simplify notation, whenever we refer in this proof to a double tower of height $p$ of the form $\left (a_{r,j}^{(k)}\right )$, it is to be understood that $r=0,1$ and $j=0,1,\ldots,p-1+r$ and $k$ has suitable range (which may be $0,1,\ldots,d_J$ or $0,1,\ldots,d_{A/J}$, depending on the context). Parts of the proof mirror closely that of Theorem \ref{Thm:extensions-finite-group}, and those are mostly omitted here. 

We denote by $\pi\colon  \A \to \A/\J$ the quotient map, and by $\overline{\alpha}$ the
quotient action on $\A/\J$. Let $F \subseteq \A$ be a finite subset, let $\eps>0$
and let $p$ be a fixed positive integer. We assume without loss of generality that $\|a\|\leq 1$ for all $a \in F$. Pick a $(d_{\A/J},\pi(F),\eps/7)$-double tower  
$(b_{r,i}
^{(k)})$ in $\A/\J$ of height $p$. Using Remark \ref{Rmk:single-auto-ortho}, we  assume without loss of generality that 
$b_{r,i}^{(k)}b_{s,j}^{(k)} =0$ for all $k=0,1,\ldots,d$ and all $(r,i) \neq (s,j)$. 
Since the cone over $\C^n$ is projective, there are $x_{r,i}^{(k)}\in \A$ for the corresponding indices $r$, $i$ and $k$ such that 
 $0\leq x_{r,i}^{(k)} \leq 1$ and $\pi(x_{r,i}^{(k)}) =
b_{r,i}^{(k)}$, and moreover 
$x_{r,i}^{(k)}x_{s,j}^{(k)} = 0$ whenever $(r,i) \neq (s,j)$. 

If $(e_{\lambda})_{\lambda \in \Lambda}$ is an approximate identity for $\J$ then, for any $a \in F$,
$$
\lim_{\lambda}\|[x_{r,i}^{(k)},a](1-e_{\lambda})\| < \eps/7
\, .
$$
Use Lemma \ref{Lemma:almost-fixed-approximate-id} to choose a quasicentral approximate identity $(e_{\lambda})_{\lambda \in \Lambda}$ for $J$ such that
$\|\alpha(e_{\lambda}) - e_{\lambda}\|<\eps/7$ for all $\lambda \in \Lambda$. By taking 
an element far enough out in this approximate identity, we can find a positive contraction $q \in \J$ which satisfies the following conditions.
\begin{enumerate}
\item 
\label{3.5-item-1}
$\displaystyle \|qa - aq\|<\frac{\eps}{21(2p+1)(d_{\A/\J}+1)}$ for all $a$ in
$$
F \cup
\left \{\alpha\left (x_{r,i}^{(k)}\right ),\alpha\left (\sqrt{x_{r,i}^{(k)}}\right ) \mid \; (r,i) \; \mathrm{as} \; \mathrm{above} \; \mathrm{and} \;
k=0,1,\ldots,d_{\A/\J} \right \}
\, .
$$
\item $\left \|[x_{r,i}^{(k)},a](1-q) \right\|<\eps/7$  for all $a \in F$.
\item $\|\alpha(q) - q\|<\eps/7$.
\item $\displaystyle \left \| \left (1-\sum_{k=0}^{d_{\A/\J}}
\left ( 
\sum_{i=0}^{p-1} x_{0,i}^{(k)} + 
\sum_{i=0}^p x_{1,i}^{(k)}
\right )
\right )a(1-q)\right \|<\eps/7$ for all $a \in F$.
\item $\left \|\left (\alpha(x_{r,i}^{(k)}) - x_{r,i+1}^{(k)} \right )a(1-q) \right \| <\eps$
for $r=0,1$, for $i=0,1,\ldots,p-2+r$, for $k=0,1,\ldots,d_{\A/\J}$, and for all $a \in F$.
\item $\left \|\left (\alpha(x_{0,p-1}^{(k)} + x_{1,p}^{(k)} ) - (x_{0,0}^{(k)}
+x_{1,0}^{(k)}) \right)a(1-q) \right \| <\eps/7$ for $r=0,1$, for $i=0,1,\ldots,p-2+r$,
for $k=0,1,\ldots,d_{\A/\J}$, and for all $a \in F$.
\item $\left \|[x_{r,i}^{(k)},x_{s,j}^{(l)}]a(1-q) \right \| < \eps/7$ for all indices $(r,i),(s,j)$ as above and 
for $k,l =0,1,\ldots,d_{\A/\J}$.
\end{enumerate}

Now,  fix a finite set $\widetilde{F} \subseteq \J$ such that for all $a \in F$, all 
indices $r,i$ as above, and all $k=0,1,\ldots,d_{\A/J}$, we have $\textrm{dist}([x_{r,i}^{(k)},a],\widetilde{F})<\eps/7$. Let
$$
F_{\J} = \widetilde{F} \cup \{qa \mid a \in F\} \cup \{q\} \cup \{qx_{r,i}^{(k)}
\mid r=0,1, i=0,1,\ldots,p-1+r, k=0,1,\ldots,d_{\A/\J}\}
\, . 
$$
Set $M = \max \{\|f\| \mid f \in F_{\J}\} $. Choose $\delta \in (0,\eps/42)$ such that 
whenever $\B$ is a $C^*$-algebra and $b,c \in \B$ satisfy $0 \leq b \leq 1$, $\|c\|\leq M$, and $\|bc-cb\|<\delta$, then 
$$
\|b^{1/2}c - cb^{1/2}\| < \frac{\eps}{21(2p+1)(d_{\J}+1)} \, .
$$
Choose a system $\left ( y_{r,i}^{(k)} \right
)_{r,i,k}$ of $(d_{\J},F_{\J},\delta)$-double towers of height $p$  in $\J$. 
Using Remark \ref{Rmk:single-auto-ortho}, we can assume that $y_{r,i}^{(k)}y_{s,j}^{(k)} = 0$ for $k=0,1,\ldots,d_{\J}$ and all distinct pairs of indices
$(r,i) \neq (s,j)$ as above.
For $k=0,1,\ldots,d_{\A/\J}$ and the indices $(r,i)$ as above, we then have
$$
\| [\sqrt{y_{r,i}^{(k)}},a] \| < \frac{\eps}{21(2p+1)(d_{\J}+1)}
$$
for all $a \in F_{\J}$. 
 Set $z_{r,i}^{(k)} = \sqrt{y_{r,i}^{(k)}}q\sqrt{y_{r,i}^{(k)}}$.

Using arguments similar to ones in the proof of Theorem \ref{Thm:extensions-finite-group}, one checks that:
\begin{enumerate}
\setcounter{enumi}{7}
\item $\|\left ( \alpha(z_{r,i}^{(k)}) - z_{r,i+1}^{(k)} \right ) a\|<\eps$ for $r=0,1$, for $i=0,1,\ldots,
p-2+r$, for $k=0,1,\ldots,d_{\A/\J}$, and for all $a \in F$. (See (\ref{2.10-E1}) in the proof of Theorem \ref{Thm:extensions-finite-group} and use $\|\alpha(q) - q\|<\eps/7$.)
\item $\|\left ( \alpha(z_{0,p-1}^{(k)} + z_{1,p}^{(k)}) - (z_{0,0}^{(k)} +
z_{1,0}^{(k)}) \right )a \|<\eps$ for $k=0,1,\ldots,d_{\J}$ and all $a \in F$. (This is similar to the previous item.)
\item $\|[z_{r,i}^{(k)},a]\|<\eps$ for $k=0,1,\ldots,d_{\J}$, for all $(r,i)$ as above, and for all $a \in F$. (See (\ref{2.10-E7}) in the proof of Theorem \ref{Thm:extensions-finite-group}.) 
\item $\displaystyle \left  \|
\sum_{k=0}^{d_{\J}} \left (
\sum_{i=0}^{p-1} z_{0,i}^{(k)} + \sum_{i=0}^{p}  z_{1,i}^{(k)} \right ) - q \right \|<\eps$. (See (\ref{2.10-4X23_1}) in the proof of Theorem \ref{Thm:extensions-finite-group}.) 
\item $\left \|\left [z_{r,i}^{(k)},z_{s,j}^{(l)} \right ] \right \|<\eps$ for all indices $r,i,s,j,k,l$ as above. (See (\ref{2.10-E2}) in the proof of Theorem \ref{Thm:extensions-finite-group}.) 
\end{enumerate}

For $k=0,1,\ldots,d_{\A/\J}$ and for all indices $(r,i)$ as above, set 
$$
\widetilde{x}_{r,i}^{(k)} = \sqrt{x_{r,i}^{(k)}}(1-q)\sqrt{x_{r,i}^{(k)}}
\, .
$$
We have, using (\ref{3.5-item-1}),  
$$
\|\widetilde{x}_{r,i}^{(k)} -  x_{r,i}^{(k)}(1-q) \| <
\frac{\eps}{21(2p+1)(d_{\A/\J}+1)}
\, ,
$$ 
so by an argument similar to that for (\ref{2.10-E5}) in the proof of Theorem \ref{Thm:extensions-finite-group}, and using $\|\alpha(q)-q\|<\eps/7$, we get
$$
\left \|\left ( \alpha(\widetilde{x}_{r,i}^{(k)}) - \widetilde{x}_{r,i+1}^{(k)} \right ) a \right \| < \eps  
$$ 
for $r=0,1$, for $i=0,1,\ldots,p-2+r$, for $k=0,1,\ldots,d_{\A/\J}$, and for all $a \in F$. Likewise,
$$
\left \|\left ( \alpha(\widetilde{x}_{0,p-1}^{(k)} + \widetilde{x}_{1,p}^{(k)})) -
(\widetilde{x}_{0,0}^{(k)}+\widetilde{x}_{1,0}^{(k)})\right )a \right \| < \eps  
\, .
$$ 
An argument similar to that for (\ref{2.10-x,a}) in Theorem \ref{Thm:extensions-finite-group} gives
$$
\left \| \left [\widetilde{x}_{r,i}^{(k)} , a \right ] \right \| < \eps
$$
for $k=0,1,\ldots,d_{A/J}$, appropriate indices $(r,i)$, and $a \in F$.
An argument like that for (\ref{2.10-4X23_2}) in the proof of Theorem \ref{Thm:extensions-finite-group} gives
$$
 \left \| \left (\sum_{k=0}^{d_{\A/\J}} \left ( \sum_{i=0}^{p-1} \widetilde{x}_{0,i}^{(k)} + \sum_{i=0}^p \widetilde{x}_{1,i}^{(k)} \right ) \right )a
- a(1-q)\right \|<\eps
$$
for all $a \in F$. For the approximate commutation condition, we have
(see (\ref{2.10-E3}) in the proof of Theorem \ref{Thm:extensions-finite-group}):
$$
\left \| \left [\widetilde{x}_{r,i}^{(k)},\widetilde{x}_{s,j}^{(l)} \right ]a\right \| \leq
\left \|(1-q)\left [x_{r,i}^{(k)},x_{s,j}^{(l)} \right ] a (1-q) \right \|  + \frac{\eps}{7} + \frac{4\eps}{21} < \eps \, .
$$
Now, for all applicable indices, we have (see (\ref{2.10-E4}) in the proof of Theorem \ref{Thm:extensions-finite-group}):
$$
\left \| \left [\widetilde{x}_{r,i}^{(k)},z_{s,j}^{(l)} \right ]a \right \| \leq
\left \| \left [qx_{r,i}^{(k)},y_{s,j}^{(l)} \right ] \right \| \cdot \|(1-q)a\| + \frac{4\eps}{21} + 4\delta < \eps 
\, .
$$
Finally, for any $a \in F$ we have (see (\ref{2.10-E6}) in the proof of Theorem \ref{Thm:extensions-finite-group}): 
$$
  \left \| \left (\sum_{k=0}^{d_{\A/\J}}\sum_{r,i} \widetilde{x}_{r,i}^{(k)}  +
\sum_{k=0}^{d_{\J}}\sum_{s,j} z_{s,j}^{(k)} \right) a - a\right \|  < \eps
\, .
$$

For $k=0,1,\ldots,d_{\A/\J} + d_{\J} + 1$, and all indices $(r,i)$ as above, set
$$
f_{(r,i)}^{(k)} = \left \{ 
\begin{matrix}
\widetilde{x}_{(r,i)}^{(k)} & \mid & k=0,1,\ldots,d_{\A/J} \\
z_{(r,i)}^{(k-d_{\A/\J} -1 )} & \mid & k=d_{\A/\J} + 1,d_{\A/\J} + 2,\ldots,d_{\A/\J} + d_{\J} + 1   
\, .
\end{matrix}
\right .
$$

Then the family $\left (f_{r,i}^{(k)}\right)$ for $k=0,1,\ldots,d_{\A/\J}+d_{\J}+1$ is a Rokhlin system of 
$(d_{\A/\J}+d_{\J}+1,F,\eps)$-double towers of height $p$. This completes the proof that  $\dimrokct(\alpha) \leq
d_{\A/\J} + d_{\J} + 1$, as required.
\end{proof}

\section{Obstructions to finite Rokhlin dimension}
\label{sec:obstructions}

The purpose of this section is to find a $K$-theoretic obstruction for an action
of a compact Lie group to have the $X$-Rokhlin property (and in particular,
finite Rokhlin dimension with commuting towers in the finite group case). This
will be stated in Corollary \ref{C_3410_IGAnnXR}. This obstruction uses
equivariant $K$-theory, viewed as a module over the representation ring. Using
this obstruction and the generalization of the Atiyah-Segal completion theorem
to $C^*$-algebras from \cite{Ph4}, we can show, for instance (Theorem
\ref{T_3410_NoXRkhOI}) that there are no actions of any nontrivial finite
group on the Jiang-Su algebra $\Zh$ or on the Cuntz algebra $\Oh_{\infty}$ with
finite Rokhlin dimension with commuting towers.

Let $G$ be a compact group,
let $\A$ be a unital $C^*$-algebra,
and let $\alpha \colon G \to \aut (\A)$
be an action of $G$ on~$\A.$
We let $R (G)$ be its representation ring,
the Grothendieck group made from the
unitary equivalence classes
of finite dimensional unitary representations of~$G$
with product given by tensor product of representations.
See~\cite{Sg1} for an extensive discussion of this ring.
If $V$ is a finite dimensional unitary representation space of~$G,$
we denote by $[V]$ its class in $R (G).$
We let $I (G)$ be the augmentation ideal of $R (G),$
that is, the kernel of the map $R (G) \to \Z$
which sends $[V]$ to $\dim (V)$ for every
finite dimensional unitary representation space of~$G.$

We take $K_0^G (A)$ to be defined following
Definition 2.4.2 and Corollary 2.4.5 of~\cite{Ph1},
except using projections and invariant partial isometries
instead of idempotents and algebraic
invariant Murray-von Neumann equivalence.
That is,
we consider the Grothendieck group made from the
semigroup of invariant Murray-von Neumann equivalence classes
of $G$-invariant projections in algebras $B (V) \otimes A,$
in which $V$ is a finite dimensional unitary representation space of~$G,$
with action by conjugation by the representation.
Projections $p \in B (V) \otimes \A$
and $q \in B (W) \otimes \A$
are invariantly Murray-von Neumann equivalent
if there is a $G$-invariant element $s \in B (V, W) \otimes A,$
with its obvious action of~$G$
(using the representations on both $V$ and~$W$)
such that $s^* s = p$ and $s s^* = q.$
For a $C^*$-algebra $\A$, this gives the same group defined there, by \cite[Proposition 2.4.11(2)]{Ph1}.

We warn that it is {\emph{not}} enough to find
an equivariant isomorphism
$\varphi \colon B (V) \to B (W)$
such that $(\varphi \otimes \id_{\A} ) (p) = q.$
The existence of such an isomorphism $\varphi$
does {\emph{not}} imply that $[p] = [q]$ in $K_0^G (\A).$

The group $K_0^G (\A)$ is an $R (G)$-module in a natural way. See
\cite[Definition 2.2.2]{Ph1} and the discussion that follows.
We recall (\cite[Remark 2.4.6]{Ph1})
that if $p \in B (V) \otimes A$ is an invariant projection
and $W$ is a finite dimensional unitary representation space of~$G,$
then $1 \otimes p \in B (W) \otimes B (V) \otimes \A$
represents $[W] \cdot [p].$

\begin{Prop}\label{P_3410_RGOnA}
Let $G$ be a compact group,
let $\A$ and $\Ch$ be unital $C^*$-algebras,
and let $\alpha \colon G \to \aut (\A)$
and $\gamma \colon G \to \aut (\Ch)$
be actions of $G$ on $\A$ and~$\Ch.$
Let $\sigma \in R (G).$
If $\A$ admits  an
approximate equivariant central unital homomorphism from
$\Ch$ in the sense of Definition \ref{Def:ecm-map},
and $\sigma \cdot [1] = 0$ in $K_0^G (\Ch),$
then $\sigma \cdot \eta = 0$ for every $\eta \in K_*^G (\A).$
\end{Prop}

\begin{proof}
Applying Lemma \ref{L_3410_TPrdCt} with $\B = C (\T)$
with the trivial action,
and using $K_*^G (\A) \cong K_0^G (C (\T) \otimes \A),$
we see that it suffices to consider $\eta \in K_0^G (\A).$
We need only prove the result for $\eta$ in a generating
set for $K_0^G (\A).$
Thus, we may assume that there is
a finite dimensional unitary representation space $W$ of~$G$
and a $G$-invariant projection $e \in B (W) \otimes \A$
such that $\eta = [e].$

Choose finite dimensional unitary representation spaces $V_1$ and $V_2$ of~$G$
such that $\sigma = [V_1] - [V_2].$
Then $1 \otimes 1 \in B (V_1) \otimes \Ch$
and $1 \otimes 1 \in B (V_2) \otimes \Ch$
have the same class in $K_0^G (\Ch).$
Therefore there is a finite dimensional unitary representation space $V_0$
of~$G$
such that,
with $V = V_0 \oplus V_1 \oplus V_2$
and with $p_0, p_1, p_2 \in B (V)$
being the projections onto $V_0,$ $V_1,$ and $V_2,$
there is a $G$-invariant partial isometry $s \in B (V) \otimes \Ch$
with
\begin{equation}\label{Eq_3410_MvN}
s^* s = (p_0 + p_1) \otimes 1
\; \textrm{ and } \;
s s^* = (p_0 + p_2) \otimes 1.
\end{equation}

Let $\varphi \colon B (V) \otimes \A \to  B(V) \otimes B(W) \otimes \A$
be the equivariant homomorphism
such that $\varphi (d \otimes a) = d \otimes 1 \otimes a$
for all $d \in B(V)$ and $a \in \A.$
Identify $B (V)$ with $M_m$ and $B (W)$
with $M_n,$
and choose $F \subseteq \A$ to consist of all the matrix entries
of~$e$
and $F_0 \subseteq \C$ to consist of all the matrix entries
of~$s.$ For sufficiently small $\eps>0$ (described below), 
choose $Q_0$ as in Definition \ref{Def:ecm-map}. As explained in Remark \ref{Rmk:ecm-map-equiv}, we may assume that $Q_0$ is $G$-equivariant. 
Set $Q = \id_{B(V)} \otimes Q_0$. 
Then $(\varphi \circ Q) (s)$ will approximately commute
with $1_{B (W)} \otimes e.$
Set $t = (\varphi \circ Q) (s) (1_{B (W)} \otimes e).$
Then $t$ is $G$-invariant.
We have
$$
(\varphi \circ Q) (p_0 + p_1) \cdot (1_{B (W)} \otimes e)
  = [(p_0 + p_1) \otimes 1_{B (V)} \otimes 1 \big]
        (1_{B (W)} \otimes e)
  = (p_0 + p_1) \otimes e
$$
and similarly
$$
(\varphi \circ Q) (p_0 + p_1) \cdot (1_{B (W)} \otimes e)
  = (p_0 + p_2) \otimes e.
$$

If $\eps > 0$ is sufficiently small,
using~(\ref{Eq_3410_MvN}), 
approximate multiplicativity of~$Q$,
and the fact that
$$
\big\| (\varphi \circ Q) (s) (1_{B (W)} \otimes e)
       - (1_{B (W)} \otimes e) (\varphi \circ Q) (s) \big\|
$$
is small,
we will get
$$
\| t^* t - (p_0 + p_1) \otimes e \| < 1
\; \textrm{ and } \;
\| t t^* - (p_0 + p_2) \otimes e \| < 1.
$$
A standard functional calculus argument therefore gives
an invariant element $v \in B (V) \otimes B (W) \otimes \A$
such that
\[
v^* v = (p_0 + p_1) \otimes e
\; \textrm{ and } \;
v v^* = (p_0 + p_2) \otimes e.
\]
These relations imply that
$([V_0] + [V_1]) [e] = ([V_0] + [V_2]) [e]$
in $K_0^G (\A).$
Therefore $\sigma [e] = [V_1] \cdot [e] - [V_2] \cdot [e] = 0.$
\end{proof}

We don't require Lie groups to be connected.
In particular,
all finite groups are compact Lie groups.

\begin{Cor}\label{C_3410_IGAnnXR}
Let $G$ be a compact Lie group (not necessarily connected),
and let $X$ be a compact free $G$-space.
Then there is $n \in \N$ such that,
for every unital $C^*$-algebra $\A$
and every action $\alpha \colon G \to \aut (\A)$
which has the $X$-Rokhlin property,
we have $I (G)^n K_*^G (\A) = 0.$
\end{Cor}

\begin{proof}It follows from \cite[Proposition 4.3]{AS}
that there is $n \in \N$ such that $I (G)^n K^*_G (X) = 0.$
Let $\A$ be a unital $C^*$-algebra,
let $\alpha \colon G \to \aut (\A)$ be an action,
and suppose that $\alpha$ has the $X$-Rokhlin property.
Then Proposition \ref{P_3410_RGOnA}
implies that $I (G)^n K_*^G (\A) = 0.$
\end{proof}

We will need $\sigma$-$C^*$-algebras (see \cite{Ph2})
and their representable K-theory $R K_*$
(see~\cite{Ph3};
there is an easier development in \cite[Section 4]{Ph5}).
For a summary,
see the end of the introduction to~\cite{Ph4}
and the beginning of Section~2 of~\cite{Ph4}.
For a compact Lie group~$G,$
we need a suitable model of the free contractible $G$-space~$E G,$
and we follow Section~2 of~\cite{Ph4}
(which in turn follows Section~2 of~\cite{AS}).
In particular, our model for $E G$
will be countably compactly generated.
For any countably compactly generated space~$X,$
we take $C (X)$ to be the algebra of all continuous functions
from $X$ to~$\C,$
not necessarily bounded,
with the topology of uniform convergence on the sets
in the chosen
countable compact generating family.
This makes $C (X)$ a $\sigma$-$C^*$-algebra.
In particular,
$C (E G)$ is a $\sigma$-$C^*$-algebra.
If $\A$ is a $C^*$-algebra with an action $\alpha \colon G \to \aut (\A),$
then also $\A \otimes C (E G)$
(using a suitable completed tensor product)
and the fixed point algebra $[\A \otimes C (E G)]^G$
are $\sigma$-$C^*$-algebras.
If we have more than one action on~$\A,$
we will write $[\A \otimes C (E G)]^{\alpha},$
thus implicitly using the same name for the action on~$\A$
and the corresponding action on $\A \otimes C (E G).$

\begin{Notation}\label{N_3410_Compl}
Let $G$ be a compact Lie group,
let $R (G)$ be its representation ring,
and let $I (G) \lhd R (G)$ be the augmentation ideal.
For any $R (G)$-module~$M,$
we denote by $M^{\wedge}$ its $I (G)$-adic completion, that is, 
$\displaystyle M^{\wedge} =  \varprojlim M/I(G)^nM$.
(This is the Hausdorff completion.)
\end{Notation}

We will need the following version of the Atiyah-Segal Completion Theorem
for $C^*$-algebras, which is a restatement of \cite[Theorem 2.4]{Ph4}.

\begin{Thm}\label{T_3410_ASCmpl}
Let $G$ be a compact Lie group,
let $\A$ be a unital $C^*$-algebra,
and let $\alpha \colon G \to \aut (\A)$
be an action of $G$ on~$\A.$
Suppose $G$ is abelian,
and there are a finite generating set $S \subseteq {\widehat{G}}$
and $n \in \N$
such that for every $\tau \in S$
and every $k \geq n,$
we have
$$
\big\{ \eta \in K_*^G (\A) \colon (1 - \tau)^k \eta = 0 \big\}
 = \big\{ \eta \in K_*^G (A) \colon (1 - \tau)^{n} \eta = 0 \big\}.
$$
Then there is a natural isomorphism
$$
R K_* \big( [\A \otimes C (E G)]^{G} \big) \cong K_*^G (\A)^{\wedge}.
$$
\end{Thm}

\begin{Prop}\label{P_3410_Htpy}
Let $G$ be a topological group,
let $\D$ be a strongly self absorbing unital $C^*$-algebra,
and let $\alpha \colon G \to \aut (\D)$
be an action of $G$ on~$\D.$
Then the action $g \mapsto {\alpha}_g\otimes \id \colon G \to \aut (\D \otimes \D)$
is homotopic to the trivial action.
\end{Prop}

\begin{proof}
By \cite{winter-ssa-Z-stable}, any strongly self absorbing $C^*$-algebra is
$K_1$-injective. Thus, by \cite[Theorem 2.2]{dadarlat-winter}, 
there is a homotopy $t \mapsto \psi_t$
of isomorphisms $\psi_t \colon \D \otimes \D \to \D \otimes \D$,
for $t \in [0, 1],$
such that $\psi_0 = \id_{\D \otimes \D}$
and $\psi_1 (a \otimes b) = b \otimes a$
for all $a, b \in \D.$ Recall that 
$\displaystyle \D \cong \bigotimes_{n = 1}^{\infty} \D$. 
We now identify $\D \otimes \D$ with
$$
\D \otimes \bigotimes_{n = 1}^{\infty} \D
  = \bigotimes_{n = 0}^{\infty} \D,
$$
with the action
$$
g \mapsto \beta_g^{(0)} = \alpha_g \otimes \bigotimes_{n = 1}^{\infty} \id_{\D}.
$$
For $N \in \N$, let $\beta^{(N)}$ be the action 
$$
g \mapsto \beta_g^{(N)} = \left ( \bigotimes_{n = 0}^{N-1} \id_{\D} \right )
\otimes \alpha_g \otimes 
\left ( \bigotimes_{n = N+1}^{\infty} \id_{\D} \right ).
$$
Thus, we have a homotopy 
$$
(t,g) \mapsto \gamma_g^{(t)}
$$ 
for $t \in [0,1]$ and $g \in G$, of actions $\gamma^{(t)}$ of $G$ on $\displaystyle \bigotimes_{n=0}^{\infty} \D$, with $\gamma^{(0)}  = \beta^{(0)}$ and $\gamma^{(1)} = \beta^{(1)}$. For $N \in \N$ and $t \in [N,N+1]$, define an action $\gamma^{(t)}$ of $G$ on $\displaystyle \bigotimes_{n=0}^{\infty} \D$ by
$$
\gamma^{(t)}_g = \left ( \bigotimes_{n = 0}^{N-1} \id_{\D} \right )
\otimes \gamma^{(t-N)}
$$
for $g \in G$. This is a homotopy of actions which are trivial on $\displaystyle \bigotimes_{n = 0}^{N-1} \D \otimes \bigotimes_{n = N+1}^{\infty} \C 1$. An $\eps/3$ argument shows that $\displaystyle \lim_{t \to \infty}(a)= a$ for all $\displaystyle a \in \bigotimes_{n = 0}^{\infty} \D$ and $g \in G$. Thus, $\beta_g^{(0)}$ is homotopic to the trivial action. 
\end{proof}

For the following theorem we recall that any nontrivial compact Lie group
contains a nontrivial finite subgroup. If the group is itself finite, this is
tautological. Otherwise, the connected component of the identity contains a
nontrivial maximal torus, which evidently has nontrivial finite subgroups.
\begin{Thm}\label{T_3410_NoXRkhOI}
Let $G$ be a compact Lie group with more than one element,
and let $X$ be a compact free $G$-space.
\begin{enumerate}
\item \label{no-action-on-Z} There is no action $\alpha$ of $G$ on
$\Oh_{\infty}$ or $\Zh$ 
which has the $X$-Rokhlin property.
\item \label{no-action-on-UHF} If $\D$ is a UHF algebra, $p$ is
a prime number that is not a factor in the supernatural number corresponding to
$\D$, and $G$ furthermore has an element of order $p$, then there is no action
$\alpha$ of $G$ on $\D$ or $\D \otimes \Oh_{\infty}$ 
which has the $X$-Rokhlin property.
\end{enumerate}
\end{Thm}

\begin{proof}
In this proof,
we will need to consider equivariant K-theory
for different actions of the same group on the same $C^*$-algebra.
For an action $\alpha \colon G \to \aut (\A)$ of a compact
group,
we therefore write $K_*^{G, \alpha} (\A)$ for $K_*^G (\A).$

If $\A,\B$ are separable unital $C^*$-algebras, and $\alpha\colon G \to
\aut(\A)$ is an action with the $X$-Rokhlin property, then, by Lemma \ref{L_3410_TPrdCt}, the action $\alpha
\otimes \id$ on $\A \otimes_{\min} \B$ has the $X$-Rokhlin property as well.
Thus, for part (\ref{no-action-on-Z}), since $\Oh_{\infty} \cong \Zh \otimes
\Oh_{\infty}$, it suffices to show that there are no actions on $\Oh_{\infty}$
with the $X$-Rokhlin property. 
 This is purely to simplify notation, since the argument is $K$-theoretic in
nature --- the same argument below works verbatim if we replace all instances of
$\Oh_{\infty}$ with $\Zh$. Likewise, for part (\ref{no-action-on-UHF}) it suffices
to consider $\D \otimes \Oh_{\infty}$, and we can furthermore assume that $\D$ is strongly self absorbing, since we may replace it with $\bigotimes_{n=0}^{\infty} \D$, and a countable infinite tensor power of a UHF algebra is strongly self absorbing. We denote $\Oh_{\infty}$ or $\D
\otimes \Oh_{\infty}$ by $E$ when considering both parts simultaneously. 

Suppose there is such an action~$\alpha.$
Then $G$ has a nontrivial finite cyclic subgroup~$H,$
and $X$ is also a compact free $H$-space. For part (\ref{no-action-on-UHF}), we
take $H = \langle g \rangle$ where $g \in G$ is a given element of order $p$.
Therefore $\alpha |_H$ is an action of $H$
on $\E$ which has the $X$-Rokhlin property. Since we wish to show that no such
actions exist, we may assume without loss of generality that $G
\cong \Z_p$ (where $p$ is the given prime in part (\ref{no-action-on-UHF}), and
$p$ is just some prime number in part (\ref{no-action-on-Z})).
The algebra $\E$ satisfies the hypotheses
of Proposition~\ref{P_3410_Htpy}.
Therefore there is a homotopy $t \mapsto \beta^{(t)}$

of actions of $G$
on $\E \otimes \E$
such that $\beta^{(0)}_h = \alpha_h \otimes \id_{\E}$
and
$\beta^{(1)}_h
 = \id_{\E \otimes \E}$
for all $h \in H.$
According to \cite[Corollary~4.2]{Ph4},
the groups
$R K_*
 \big( [\E \otimes \E
                \otimes C (E H)]^{\beta^{(t)}} \big)$
do not depend on~$t.$

Take $t = 0.$
The action $\beta^{(0)}$ has the $X$-Rokhlin property
because $\alpha$ does.
So Corollary~\ref{C_3410_IGAnnXR}
provides an $n \in \N$ such that
$$
I (H)^n
 K_*^{H, \beta^{(0)}}
  (\E \otimes \E) = 0.
$$
The hypothesis 
of Theorem~\ref{T_3410_ASCmpl}
is satisfied,
because  ${\widehat{G}}$ is finite and for every $\tau \in {\widehat{G}}$ and every $k \geq n,$
we have
\[
\big\{ \eta \in K_*^{G, \beta^{(0)}} (\E
   \otimes \E) \colon (1 - \tau)^k \eta = 0 \big\}
  = K_*^G (\E
   \otimes \E).
\]
So Theorem~\ref{T_3410_ASCmpl} implies
\[
R K_* \big( [\E \otimes \E
               \otimes C (E G)]^{\beta^{(0)}} \big)
  \cong K_*^{G, \beta^{(0)}} (\E
     \otimes \E)^{\wedge}.
\]
In the $I (G)$-adic topology,
this group is discrete,
since it is annihilated by $I (G)^n.$

Now take $t = 1.$
We verify
the hypothesis 
of Theorem~\ref{T_3410_ASCmpl}. 

Let $\tau$ be a generator for the dual group ${\widehat{G}} \cong \Z_p$. Using
the trivial action,
$$
K_*^G ( \E )
 \cong R (G) \otimes_{\Z} K_0(\E) 
 \cong K_0(\E) [ \tau ] / \langle 1 - \tau^p \rangle,
$$ 
which we write as
$$
\left \{ \sum_{j=0}^{p-1} n_j\tau^j \colon n_j \in K_0(\E) \right \}.
$$
Since $K_0(E)$ is a nontrivial subgroup of $\Q$ (for all possibilities for the algebra $E$), the complexification $K_*^G(E)\otimes_{\Z}\C$ is given by 
$$
\left \{ \sum_{j=0}^{p-1} n_j\tau^j \colon n_j \in \C \right \}
$$
and we can view 
$$
\left \{ \sum_{j=0}^{p-1} n_j\tau^j \colon n_j \in K_0(\E) \right \} \subset \left \{
\sum_{j=0}^{p-1} n_j\tau^j \colon n_j \in \C \right \}
$$
by viewing $K_0(\E) \subseteq \Q \subset \C$ in the usual way.
It is now a straightforward computation to verify that for all $n$, thinking of $1-\tau$ as a linear transformation on $K_*^G(E)\otimes_{\Z}\C$, we have  
$$
\ker(1-\tau)^n = \ker(1-\tau) = \left \{ \sum_{j=0}^{p-1} n_j\tau^j \colon n_j =
n_0 \textrm{ for all } j \right \} .
$$
From this description it follows that the same holds without complexifying the
module. 
Thus,
the hypothesis 
of Theorem~\ref{T_3410_ASCmpl} holds.

We now consider the two parts of the theorem separately.

For part (\ref{no-action-on-Z}), Theorem~\ref{T_3410_ASCmpl} implies
\[
R K_* \big( [\Oh_{\infty} \otimes \Oh_{\infty}
                 \otimes C (E G)]^{\beta^{(1)}} \big)
  \cong K_*^{G, \beta^{(1)}} (\Oh_{\infty}
                 \otimes \Oh_{\infty})^{\wedge}
  \cong R (G)^{\wedge}.
\]
This group is not discrete in the $I (G)$-adic topology.
Since
\[
R K_* \big( [\Oh_{\infty} \otimes \Oh_{\infty}
               \otimes C (E G)]^{\beta^{(0)}} \big)
  \cong R K_* \big( [\Oh_{\infty} \otimes \Oh_{\infty}%
               \otimes C (E G)]^{\beta^{(1)}} \big),
\]
we have a contradiction.

For part (\ref{no-action-on-UHF}), we also need to show that $K_*^G ( \D \otimes
\Oh_{\infty} )^{\wedge}$ is not discrete. We first claim that there is a nonzero element $\sigma \in \Z[\tau]/\langle 1-\tau^p \rangle \cong R(G)$ such that $(1-\tau)^p = p\sigma$. 
If $p=2$, one checks that
$(1-\tau)^2 = 2(1-\tau)$. If $p$ is odd, 
since $(-\tau)^p = -1$, we can take
$$
\sigma = \frac{1}{p} \sum_{k=1}^{p-1} (-1)^k {{p}\choose{k}} \tau^k .
$$
(Since $p$ is prime, the coefficients really are in $\Z$.)
This proves the claim.

Since $1-\tau$ generates $I(G)$, it
follows that if $n\geq kp$ then
$$
I (G)^n K_*^G ( \D\otimes \Oh_{\infty} ) \subseteq p^k K_0^G(\D\otimes
\Oh_{\infty}) .
$$

We claim that multiplication by $1-\tau$ is injective on $(1-\tau)R(G)\otimes_{\Z}K_0(\D\otimes \Oh_{\infty})$. It suffices to prove this with $\C$ in place of $K_0^G(\D\otimes \Oh_{\infty})$, thus on 
$\C[\tau]/\langle 1-\tau^p \rangle$. Here, multiplication by $\tau$ is the cyclic shift with respect to the basis $1,\tau,\ldots,\tau^{p-1}$, so the claim is straightforward to check.

For all $m \in \N$, multiplication by $p$ is not invertible on $p^mR(G) \otimes_{\Z} K_0(\D \otimes \Oh_{\infty})$. So the sequence $\left ( 
I(G)^n K_*^G(\D \otimes \Oh_{\infty})
\right )_{n \in \N}$ does not stabilize.
Therefore $K_*^G ( \D \otimes \Oh_{\infty}
)^{\wedge}$ is not discrete.
Thus, no such action $\alpha$ exists.
\end{proof}

\begin{Rmk} 
The argument for part (\ref{no-action-on-UHF}) of Theorem \ref{T_3410_NoXRkhOI}
breaks down if $p$ does appear as a factor in the supernatural number, since
then multiplication by $p$ is invertible. Indeed, there are actions
of $\Z_p$ on the $p^{\infty}$ UHF algebra with the Rokhlin property. (For instance, let $v \in M_p$ be a cyclic permutation matrix of order $p$. Then it is easy to check that the action of $\Z_p$ given by the order $p$ automorphism 
$\bigotimes_1^{\infty} \mathrm{Ad}(v)$ of  $\bigotimes_1^{\infty}M_p$ has the Rokhlin property. See \cite[Example 3.2]{Izumi-I}.)
\end{Rmk}

\begin{Cor}
Let $G$ be a nontrivial finite group.
\begin{enumerate}
\item There is no action of $G$ on $\Zh$ or $\Oh_{\infty}$ which has finite Rokhlin dimension with commuting towers.
\item Let $p$ be a prime number. Let $\D$ be a UHF algebra which does not have $p$ as a factor in its corresponding supernatural number. If $G$ has an element of order $p$, then there is no action of $G$ on $\D$ or $\D \otimes \Oh_{\infty}$  which has finite Rokhlin dimension with commuting towers.
\end{enumerate}
\end{Cor}

We conclude with an analog of Theorem \ref{T_3410_NoXRkhOI} for actions of connected Lie groups. The key fact we use concerning such groups is that they contain a subgroup isomorphic to the circle group $\T$. The proof is similar to that of Theorem \ref{T_3410_NoXRkhOI}, so we only outline the parts that need to be changed.
Recently, Eusebio Gardella showed with different methods that the statement of Theorem \ref{Thm:no-circle-action} holds in greater generality: there are no actions of infinite Lie groups with finite Rokhlin dimension with commuting towers on $C^*$-algebras with just one vanishing $K$-group. See \cite[Corollary 4.18]{gardella}.
\begin{Thm}
\label{Thm:no-circle-action}
Let $G$ be a compact group which contains a closed subgroup isomorphic to $\T$. Let $X$ be a compact free $G$-space. Then there is no action of $G$ on any UHF algebra with the $X$-Rokhlin property. 
\end{Thm}
\begin{proof}
Since the $X$-Rokhlin property passes to subgroups, we may assume without loss of generality that $G = \T$. Let $\D$ be a UHF algebra. We wish to show that there is no action of $G$ on $\D$ with the $X$-Rokhlin property. We may assume without loss of generality that $\D$ is the universal UHF algebra, since if we have an action $\alpha$ of $G$ on $\D$ with the $X$-Rokhlin property, we can tensor it by the identity action of $G$ on the universal UHF algebra, and by Lemma \ref{L_3410_TPrdCt} the tensor product action has the $X$-Rokhlin property as well. 
The dual group $\widehat{G} \cong \Z$ is singly generated. We have $R(G) \cong \Z[\tau,\tau^{-1}]$, so $R(G) \otimes_{\Z} K_0(\D) \cong \Q[\tau,\tau^{-1}]$. One checks that multiplication by $1-\tau$ is injective on $\Q[\tau,\tau^{-1}]$ and the subgroups
$(1-\tau)^n \Q[\tau,\tau^{-1}]$ do not stabilize. The rest of the proof is similar to that of Theorem \ref{T_3410_NoXRkhOI}. 
\end{proof}
\begin{Rmk}
The conclusion of Theorem \ref{Thm:no-circle-action} does not hold if instead of assuming that $G$ contains a copy of the circle, we only assume that $G$ is an infinite compact group. For example, it is easy to construct an action of 
$\prod_{n=1}^{\infty} \Z_2$ on the $2^{\infty}$ UHF algebra with the Rokhlin property. 
\end{Rmk}

\bibliographystyle{alpha}
\bibliography{Rokhlin-ext-bib}
\end{document}